\documentclass[12pt]{amsart}
\usepackage[colorlinks=true,urlcolor=blue, citecolor=red,linkcolor=blue,linktocpage,pdfpagelabels, bookmarksnumbered,bookmarksopen]{hyperref}
\usepackage[hyperpageref]{backref}
\usepackage{xcolor}
\usepackage{amsthm} 
\usepackage{latexsym,amsmath,amssymb}
\usepackage{accents}
\usepackage{a4wide}
\usepackage{soul}
\usepackage{mathtools} 
\usepackage{xparse} 
  
\title[Singular Set]{On the size of the singular set of minimizing harmonic maps into the 2-sphere\\ in dimension four and higher}

\author{Katarzyna Mazowiecka}
\address[Katarzyna Mazowiecka]{
Institute of Mathematics,%
University of Warsaw,
Banacha 2,
02-097 Warszawa, Poland
\newline
\&
Universit\'e catholique de Louvain, Institut de Recherche en Math\'ematique et Physique, Chemin du Cyclotron 2 bte L7.01.01, 134 8 Louvain-la-Neuve, Belgium}
\email{katarzyna.mazowiecka@uclouvain.be}

\author{Micha\l{} Mi\'{s}kiewicz}
\address[Micha\l{} Mi\'{s}kiewicz]{
Institute of Mathematics,
University of Warsaw,
Banacha 2,
02-097 Warszawa, Poland}
\email{m.miskiewicz@mimuw.edu.pl}

\author{Armin Schikorra}
\address[Armin Schikorra]{Department of Mathematics,
University of Pittsburgh,
301 Thackeray Hall,
Pittsburgh, PA 15260, USA}
\email{armin@pitt.edu}

\definecolor{indigo}{rgb}{0.29, 0.0, 0.51}

plus 6pt minus 12pt
\newcommand{\sing}{\operatorname{sing}}

\def\eps{\varepsilon}

\def\vp{\varphi}

\def\B{{B}}

\def\N{{\mathbb N}}
\def\n{{\mathcal N}}
\def\H{{\mathcal H}}

\def\S{{\mathbb S}}

\newcommand{\cH}{\mathcal H}

\newcommand{\cB}{\mathcal B}
\newcommand{\cA}{\mathcal A}
\newcommand{\bad}{{\mathsf{Bad}}}
\newcommand{\good}{{\mathsf{Good}}}

\newtheorem{theorem}{Theorem}
\newtheorem{lemma}[theorem]{Lemma}
\newtheorem{corollary}[theorem]{Corollary}
\newtheorem{proposition}[theorem]{Proposition}

\newtheorem{remark}[theorem]{Remark}
\newtheorem{definition}[theorem]{Definition}

\def\dist{{\rm dist\,}}

\newcommand{\dd}{\,\mathrm{d}}
\newcommand{\dx}{\,dx}
\newcommand{\dhn}{\,d\mathcal{H}^{n-1}}
\newcommand{\R}{\mathbb{R}}

\newcommand{\brac}[1]{\left (#1 \right )}

\newcommand{\barint}{
\rule[.036in]{.12in}{.009in}\kern-.16in \displaystyle\int }

\newcommand{\barcal}{\mbox{$ \rule[.036in]{.11in}{.007in}\kern-.128in\int $}}

\def\mvint_#1{\mathchoice
          {\mathop{\vrule width 6pt height 3 pt depth -2.5pt
                  \kern -8pt \intop}\nolimits_{\kern -3pt #1}}%
          {\mathop{\vrule width 5pt height 3 pt depth -2.6pt
                  \kern -6pt \intop}\nolimits_{#1}}%
          {\mathop{\vrule width 5pt height 3 pt depth -2.6pt
                  \kern -6pt \intop}\nolimits_{#1}}%
          {\mathop{\vrule width 5pt height 3 pt depth -2.6pt
                  \kern -6pt \intop}\nolimits_{#1}}}

\numberwithin{theorem}{section} \numberwithin{equation}{section}

\newcommand{\aleq}{\precsim}

\let\latexchi\chi
\makeatletter
\renewcommand\chi{\@ifnextchar_\sub@chi\latexchi}
\newcommand{\sub@chi}[2]{
  \@ifnextchar^{\subsup@chi{#2}}{\latexchi^{}_{#2}}%
}
\newcommand{\subsup@chi}[3]{
  \latexchi_{#1}^{#3}%
}
\makeatother

\begin{document}

\begin{abstract}
We extend the results of our recent preprint \cite{MMS18} into higher dimensions $n \geq 4$.
For minimizing harmonic maps $u\in W^{1,2}(\Omega,\S^2)$ from $n$-dimensional domains into the two dimensional sphere we prove:
\begin{enumerate}
 \item An extension of Almgren and Lieb's linear law, namely \[\H^{n-3}(\sing u) \le C \int_{\partial \Omega} |\nabla_T u|^{n-1} \dhn;\] 
 \item An extension of Hardt and Lin's stability theorem, namely that the size of singular set is stable under small perturbations in $W^{1,n-1}$ norm of the boundary. 
\end{enumerate}
\end{abstract}

\sloppy

\maketitle
\tableofcontents
\sloppy

\section{Introduction}
In \cite{MMS18} we studied singularities of minimizers of the Dirichlet energy
\[
 \int_{\Omega} |\nabla u|^2 \dx \quad \mbox{among $u\in W^{1,2}(\Omega,\S^2)$ with fixed boundary data}
\]
for smooth, bounded domains $\Omega \subset \R^3$. We refer to \cite{MMS18} for an introduction and background, as well as a discussion of earlier works.

There are two main results in \cite{MMS18}. Firstly, we sharpen result by Hardt and Lin \cite{HL1989}, on the stability of singularities.

\begin{theorem}\label{th:stability3D}
Let $\Omega \subset \R^3$ be a bounded domain with smooth boundary. Let $s \in (\frac{1}{2},1]$, $p \in [2,\infty)$ and $sp=2$.

Assume that $u \in W^{1,2}(\Omega,\S^2)$ is the \emph{unique} minimizing harmonic map with boundary $u\rvert_{\partial \Omega} = \varphi \in W^{s,p}(\partial \Omega,\S^2)$. 

Then for any $\eps > 0$ there is a $\delta = \delta(\eps,\Omega,\varphi) > 0$ such that whenever $v$ is a minimizing harmonic map $v \in W^{1,2}(\Omega,\S^2)$ with trace $\psi := v \Big |_{\partial \Omega}$ close to $\varphi$, namely
\[
[\psi-\varphi]_{W^{s,p}(\partial \Omega)} \leq \delta,
\]
then $v$ has the same number of singularities as $u$. Moreover
\[
\|u - v\|_{W^{1,2}(\Omega)} \leq \eps.
\]
\end{theorem}
Theorem~\ref{th:stability3D} has previously been known under Lipschitz boundary assumption. See also \cite{Li18} for a related result. Theorem~\ref{th:stability3D} is sharp as was shown in \cite{MazowieckaStrzelecki}.

Secondly, we extended Almgren and Lieb's linear law, \cite[Theorem 2.12]{AlmgrenLieb1988}, to general trace spaces (whereas Almgren and Lieb had proved this previously for $W^{1,2}$-traces).
\begin{theorem}[Almgren and Lieb's linear law for trace spaces]\label{th:ll3D}
Let $\Omega \subset \R^3$ be a bounded domain with smooth boundary, $s \in (\frac{1}{2},1]$, $p \in (1,\infty)$ and $sp = 2$. Then there exist a constant $C = C(\Omega,s,p) > 0$ such that if $u \in W^{1,2}(\Omega,\S^2)$ is a minimizing harmonic map with trace $\varphi := u \Big |_{\partial \Omega}$, then 
\[
\mathcal{H}^0\{ \sing u \}\leq C\, \int_{\partial \Omega}\int_{\partial \Omega} \frac{|\varphi(x)-\varphi(y)|^p}{|x-y|^{2+sp}}\, dx\, dy.
\]
\end{theorem}
In this work we extend these results to higher dimensional domains $\Omega \subset \R^n$ for $n \geq 3$. Let us remark that for simplicity below we will make no attempt at obtaining the estimates in trace spaces $W^{s,p}$, but stick to the Sobolev space $W^{1,p}(\partial \Omega)$. An easy combination of the arguments presented here and in \cite{MMS18} will lead to them.

Our generalization of Theorem~\ref{th:stability3D} to higher dimensional domains takes the following form, see Theorem~\ref{th:global-stability}.
\begin{theorem}\label{th:stabilitynD}
Let $\Omega \subset \R^n$, $n \geq 3$, be a bounded smooth domain. Assume that for some boundary map $\varphi \in W^{1,n-1}(\partial \Omega,\S^2)$ there is a unique minimizer $u \in W^{1,2}(\Omega,\S^2)$. 

Then, for each $\varepsilon > 0$ there is a $\delta > 0$ such that if
\begin{equation}\label{eq:stabilityn}
\| \psi - \varphi \|_{W^{1,n-1}(\partial \Omega)} < \delta
\quad \Rightarrow \quad 
d_W \left( \H^{n-3} \llcorner \sing v, \H^{n-3} \llcorner \sing u \right) < \varepsilon
\end{equation}
for any minimizer $v$ with boundary data $\psi$. 

Here, $d_W$ is the $1$-Wasserstein distance, see \eqref{eq:wasserstein}, which in particular satisfies
\[
 \left | \H^{n-3} (\sing u) - \H^{n-3} (\sing v) \right | \aleq d_W \left( \H^{n-3} \llcorner \sing u, \H^{n-3} \llcorner \sing v \right) .
\]
\end{theorem}
For $n=3$ Theorem~\ref{th:stabilitynD} implies indeed Theorem~\ref{th:stability3D}: in three-dimensional domains the singular set is locally finite and any continuous map into integers is constant. In this sense, Theorem~\ref{th:stabilitynD} is the natural extension of Theorem~\ref{th:stability3D} to higher dimensions.

We also obtain a generalization of Theorem~\ref{th:ll3D} which takes the following form. See Theorem~\ref{th:almgren-lieb-high}.

\begin{theorem}\label{th:intro:al}
Let $u:\Omega\to \S^2$ be a minimizing map with $u\rvert_{\partial \Omega} = \varphi$, $\varphi \colon \partial \Omega \to \S^2$, where $\Omega\subset \R^n$ is a bounded, smooth domain. Then 
 \[
  \H^{n-3}(\sing u) \le C(n,\Omega) \int_{\partial \Omega} |\nabla \varphi|^{n-1} \dhn.
 \]
\end{theorem}
In the case of Theorem~\ref{th:intro:al} it is possible to replace $\S^2$ by any simply connected target manifold $\n$, since then Theorem~\ref{th:extensionthm} still applies (see \cite[Theorem 6.2]{HL1987}). 

As for Theorem~\ref{th:stabilitynD} the situation is more involved as our argument relies also on the classification of tangent maps. For the case when the target manifold is $\S^2$ this was obtained in \cite[Theorem 1.2]{BCL1986}. 
Such a classification is also known for targets $\S^3$, see \cite{Nakajima,LinWang2006}. In this case $\dim_{\mathcal{H}} \sing u \leq n-4$,  \cite{SU84}, thus it should be possible to extend Theorem~\ref{th:stabilitynD} in this case to an estimate of $\mathcal{H}^{n-4}(\sing u)$ following the spirit of our argument. 

Another, very challenging, question is whether the $W^{1,n-1}(\partial \Omega)$-condition in Theorem~\ref{th:intro:al} can be improved. Technically, this condition controls each singularity close to the boundary of $\Omega$ -- but only $(n-3)$-dimensional singularities appear in the estimate. So one might be tempted to believe that a $W^{1,2}$-bound is sufficient in \eqref{eq:mainestimate}. On the other hand, one might also be able to analyze each stratum of the singular set via a different Sobolev norm along the boundary.

While the proofs of the theorems above are in spirit very similar to the arguments in \cite{MMS18}, there is one main new ingredient: the following (interior) analysis of the singular set of harmonic maps by Naber and Valtorta \cite{NabVal17}. 
\begin{theorem}[{{\cite[Theorem~1.6]{NabVal17}}}]
\label{th:NabVal-general-meas-bound}
For $n \geq 3$, let $u \colon B_{2r}(x) \to \n$ be energy minimizing and \[r^{2-n} \int_{B_{2r}(x)} |\nabla u|^2 \le \Lambda.\] Then there exists a constant $C=C(n,\n,\Lambda) > 0$ such that \[\H^{n-3}(\sing u \cap B_r(x)) \le C r^{n-3}.\]
\end{theorem}

Results on the analysis of singular sets were also considered in \cite{M03,KM07}. 

\textbf{Notation.}
We denote by $B_r(x)$ the ball centered in $x$ with radius $r$. By $\R^n_+ = \R^n\cap\{x\in\R^n: x_n > 0\}$ we denote the upper half-space. $B^+_r(x)$ is given by $B^+_r(x) = B_r(x) \cap \R^n_+$. For any $\rho>0$ we write $T_\rho = B_\rho \cap \{x\in\R^n:x_n=0\}$ for the flat part and $S_\rho^+ = \partial B_\rho \cap \R^n_+$ for the curved part of the boundary of the half ball $B_\rho^+$.

For simplicity we use Greek letters $\psi, \varphi,$ etc. for boundary maps and $u$, $v$, etc. for interior maps. The letters $r, R, \rho$ will be usually reserved for the radii. We use $\nabla_T u$ for the tangential gradient of $u$, i.e., the gradient of its restriction $u |_{\partial \Omega}$.  As usual, the constant $C$ will denote a generic constant that may vary from line to line.

Throughout the paper the term \emph{minimizer} or \emph{energy minimizer} will refer to an $\S^2$-valued map minimizing the Dirichlet energy among $W^{1,2}(\Omega,\S^2)$ maps with same boundary data, unless otherwise stated.

\textbf{Acknowledgments.}
The authors would like to thank Pawe\l{} Strzelecki for suggesting extending the results of \cite{AlmgrenLieb1988} to higher dimensions. 

Financial support is acknowledged as follows
\begin{itemize}
\item National Science Centre Poland via grant no. 2015/17/N/ST1/02360 (KM)
\item National Science Centre Poland via grant no. 2016/21/B/ST1/03138 and Etiuda scholarship no. 2018/28/T/ST1/00117 (MM)
\item German Research Foundation (DFG) through grant no.~SCHI-1257-3-1 (KM, AS)
\item Daimler and Benz foundation, grant no 32-11/16 (KM,AS) 
\item Simons foundation, grant no 579261 (AS)
\item Mandat d'Impulsion scientifique (MIS) F.452317 - FNRS (KM)
\end{itemize}

\section{Tangent Maps and the interior estimates by Naber-Valtorta}
The crucial new ingredient in comparison to the original work by Almgren and Lieb \cite{AlmgrenLieb1988} are interior estimates on singular sets recently obtained by Naber and Valtorta \cite{NabVal17}.

\subsection{Tangent maps}\label{s:tangentmaps}
In this subsection we recall various facts concerning tangent maps which will be useful for future purposes. For more details we refer the interested reader to \cite[Chapter 3]{Simon1996}.

We start with the monotonicity formula see \cite[Lemma 2.4]{SU1}, \cite[Lemma 4.1]{HL1987}, or \cite[Section 2.4]{Simon1996}.

\begin{theorem}[Monotonicity formula]\label{th:monotonicityformula}
Let $\Omega \subset \R^n$ and let $u\in W^{1,2}(\Omega, \S^2)$ be a minimizing harmonic map. Then for any $0<r<R<\dist(y,\partial \Omega)$ 
 \begin{equation}\label{eq:monotonicityformula}
  R^{2-n}\int_{B_R(y)}|\nabla u|^2\dx - r^{2-n}\int_{B_r(y)}|\nabla u|^2 \dx = 2\int_{B_R(y)\setminus B_r(y)}|x-y|^{2-n}\left|\frac{\partial u}{\partial \nu}\right|^2 \dx,
  \end{equation}
  where $\frac{\partial u}{\partial \nu}$ is the directional derivative in the radial direction $\frac{x-y}{|x-y|}$.
\end{theorem}

Let $u\in W^{1,2}(\Omega,\S^2)$ be a minimizing harmonic map, $y\in\Omega$ and $\lambda>0$. We define the rescaled maps $u_{y,\lambda}\in W^{1,2}\brac{\frac 1\lambda (\Omega - y),\S^2}$ by
\[
 u_{y,\lambda}(x) : = u(y+ \lambda x).
\]
By the monotonicity formula we know that if we let $\lambda_i\searrow 0$ then $\limsup_{i\rightarrow\infty} \int_{B_r(0)}|\nabla u_{y,\lambda_i}|^2 \dx < \infty$ for all $r>0$. Thus, by the compactness theorem (cf. Theorem~\ref{th:bsc}), we obtain a subsequence $\lambda_{i_j}$ such that 
\[
 u_{y,\lambda_{i_j}} \xrightarrow{j\rightarrow \infty} \Phi \quad \text{in } W^{1,2}_{loc}(\R^n,\S^2) 
\]
and $\Phi\in W^{1,2}(\R^n,\S^2)$ is locally a minimizing harmonic map. Moreover, also by the monotonicity formula $\Phi$ is homogeneous of degree 0 (see \cite[Lemma 2.5]{SU1}), i.e., $\Phi(x) = \Phi(\lambda x)$ for all $\lambda>0$. We call such map a \emph{tangent map to $u$ at point $y$}. 

In the case $n=3$ Simon proved uniqueness of the tangent maps \cite[Section 8]{Simon1983}, but in general if we choose a different subsequence of $\lambda_i$ the limiting map may be different (see \cite{White1992}).

For $B_r(x)\subset\Omega$ we denote the rescaled energy by
\begin{equation}\label{eq:thetau}
 \theta_u(x,r): = r^{2-n}\int_{B_r(x)} |\nabla u|^2
\end{equation}
and the energy density at $y$ by 
\[
 \theta_u(x,0) := \lim_{r\searrow 0} \theta_u(x,r) = \lim_{r\searrow 0} r^{2-n}\int_{B_r(x)} |\nabla u|^2. 
\]
For any tangent map $\Phi$ the maximum of the energy density is attained at $0\in\R^n$:
\[
 \theta_\Phi(y,0)\le \theta_\Phi(0,0) \quad \text{for any } y \in \R^n.
\]
If we assume additionally that $\theta_\Phi(y,0) = \theta_\Phi(0,0)$ then we obtain
\[
 \Phi(x+\lambda y) = \Phi(x+y) \quad \text{for any } \lambda>0 \text{ and } x\in \R^n,
\]
which leads to the definition
\[
 S(\Phi):= \{y\in\R^n \colon \theta_\Phi(y,0) = \theta_\Phi(0,0)\}.
\]
Observe that for non-constant tangent map $\Phi$ we have $S(\Phi)\subset \sing \Phi$.

We introduce the notion of $k$-symmetric maps. A map $f \colon \R^n \to \S^2$ is called $k$-symmetric if $f(\lambda x) = f(x)$ for any $x \in \R^n$, $\lambda > 0$, and there exists a linear $k$-dimensional plane $L \subset \R^n$ such that $f(x+y) = f(x)$ for any $x \in \R^n$, $y \in L$. The space of such functions will be denoted by $\mathrm{sym}_{n,k}$. 

Next we observe
\[
 y\in \sing u \Longleftrightarrow\dim S(\Phi) \le n-1 \quad \text{ for every tangent map $\Phi$ of $u$ at $y$}.
\]
We define for all $j\in \{0,\ldots,n-1\}$
\[
\begin{split}
 S_j &:= \{y\in \sing u \colon \dim S(\Phi) \le j \text{ for all tangent maps $\Phi$ of $u$ at $y$} \}\\
 &\le \{y\in \sing u \colon \text{ no tangent map of $u$ at $y$ belongs to $\mathrm{sym}_{n,j+1}$}\}.
 \end{split}
\]
\begin{equation}
\label{eq:Sjdim}
 \dim_{\mathcal{H}} (S_j) \leq j
\end{equation}
and in particular from the regularity result $\dim_\H (\sing u) \le n-3$, see \cite[Theorem II]{SU1}. This gives us the stratification
\[
 S_0 \subset S_1 \subset \ldots \subset S_{n-4}\subset S_{n-3} = S_{n-2} = S_{n-1} = \sing u.
\]

We will be mainly interested in the top-dimensional part of the singular set, so for this purpose we define
\[
 \sing_* u = S_{n-3} \setminus S_{n-4}.
\]

We also recall the classification of tangent maps by Brezis--Coron--Lieb.
\begin{theorem}[{{\cite[Theorem~1.2]{BCL1986}}}]\label{th:BCLclassification}
In the case $n=3$ every nonconstant tangent map must have the form 
$\mathcal R \brac{\frac{x}{|x|}}$ for a orthogonal rotation $\mathcal R$ of $\R^3$.
\end{theorem}

We will use the symbol $\Psi \colon \R^n \to \S^2$ to denote the map
\begin{equation}
\label{eq:def-psi}
\R^3 \times \R^{n-3} \ni (x',x'')
\xmapsto{\quad \Psi \quad}
\frac{x'}{|x'|} \in \S^2.
\end{equation}
We note that the map $\Psi$ belongs to $\mathrm{sym}_{n,k}$ for all $k=0,1,\ldots,n-3$ but not to $\mathrm{sym}_{n,n-2}$. 
Its energy density will be denoted by 
\begin{equation}
\label{eq:def-Theta}
\Theta := \int_{B_1} |\nabla \Psi|^2 \dx.
\end{equation}

\begin{corollary}\label{co:uniquetangent}
Suppose $u\in W^{1,2}(\Omega, \S^2)$ is a minimizing harmonic map and $y\in \sing_*u$, then up to isometries of $\R^n$ the only tangent map of $u$ at $y$ is $\Psi$. In particular the density of a tangent map to $u$ at a point from $\sing_* u$ is constant.
\end{corollary}

\begin{proof}
Let $\Phi$ be any tangent map to $u$ at $y$. By definition, since $y\in \sing_*u$ it means that $\Phi$ is $(n-3)$-symmetric, thus if $(x',x'')\in \R^3\times \R^{n-3}$ we have $\Phi(x',x'') = \Phi_0(x') = w\brac{\frac{x'}{|x'|}}$, where the last equality results from the 0-homogeneity of $\Phi$. Now, by \cite[Lemma 5.2]{SU1} we know that $\Phi_0$ is also a locally minimizing harmonic map. By Theorem~\ref{th:BCLclassification} we know that up to an orthogonal rotation $\Phi_0 = \frac{x'}{|x'|}$. See also \cite[Corollary 2.2]{HLB4}. 
\end{proof}

\subsection{Refined estimates by Naber and Valtorta}
Here we discuss the results of Naber and Valtorta \cite{NabVal17} needed in the sequel. A simplified presentation of these is available in their later article \cite{NabVal18}.

The main ingredient is Theorem~\ref{th:NabVal-general-meas-bound}.  
%
%
In the special case of $\n = \S^2$, uniform boundedness of minimizers, Theorem~\ref{th:uniformboundedness}, implies that the energy assumption is redundant. 

\begin{corollary}
\label{co:NabVal-meas-bound}
If $u \colon B_{2r} \to \S^2$ is energy minimizing then $\H^{n-3}(\sing u \cap B_r) \le C r^{n-3}$ with some constant $C(n) > 0$.

In particular, whenever $\Omega' \subset \subset \Omega$ and $u$ is a minimizing harmonic map on $\Omega$, then
\[
 \mathcal{H}^{n-3}(\sing u \cap \Omega') < \infty.
\]
\end{corollary}

In order to prove the stability theorem, Theorem~\ref{th:stabilitynD}, one needs more subtle measure estimates. Note that for the tangent map $\Psi$, the singular set is an $(n-3)$-plane and so $\H^{n-3}(\sing \Psi \cap B_r) = \omega_{n-3} r^{n-3}$. If $u$ is close to $\Psi$, one could expect its singular set to have similar measure, see Lemma~\ref{lem:local-stability}. To this end, we will need two more results, which are essential ingredients of \cite{NabVal17}. 

To state them, we first recall the definition of Jones' height excess $\beta$-numbers. Choosing a Borel measure $\mu$ in $\R^n$, a dimension $0 < k < n$ and an exponent $p \ge 1$, we can define for each ball $B_r(x)$ 
\[
\beta_{\mu,k,p} := \inf_{L} \left( r^{-k-p} \int_{B_r(x)} \dist(y,L)^p \dd \mu (y) \right)^{1/p},
\]
where the infimum is taken over all $k$-dimensional affine planes $L \subset \R^n$. This measures how far the support of $\mu$ is from a $k$-dimensional plane (on the ball $B_r(x)$). However, we shall not work directly with this definition, but rather rely on the two theorems below, since they encompass all the geometric information we need.

The first theorem is a general geometric result that gives sharp measure estimates. 

\begin{theorem}[{Rectifiable Reifenberg {\cite[Theorem~3.3]{NabVal17}}}]
\label{th:NabVal-reifenberg}
For every $\eps > 0$ there is a $\delta = \delta(n,\eps) > 0$ such that the following holds. Let $S \subset \R^n$ be a $\H^{k}$-measurable subset and assume that for each ball $B_r(x) \subset B_2$ 
\[
\int_{B_r(x)} \int_0^r \beta_{\mu,k,2}(y,s)^2 \frac{\dd s}{s} \dd \mu(y) 
\le \delta r^{k}, 
\]
where $\mu$ denotes the measure $\H^{k}\llcorner S$. Then $\mu(B_1) \le (1+\eps) \omega_k$. 
\end{theorem}

As a side remark, let us note that in our application the set $S$ will satisfy the so-called Reifenberg condition and so one could work with the $W^{1,p}$-Reifenberg theorem \cite[Theorem~3.2]{NabVal17} instead. 

\begin{theorem}[{$L^2$-best approximation {\cite[Theorem~7.1]{NabVal17}}}]
\label{th:NabVal-best-approx}
For every $\eps > 0$ there are $\delta(n,\eps) > 0$ and $C(n,\eps) > 0$ such that the following holds. If $u \colon B_{10} \to \S^2$ is energy minimizing, 
\begin{align*}
\dist_{L^2(B_{10})} (u, \ \mathrm{sym}_{n,0}  ) & \le \delta, \\
\dist_{L^2(B_{10})} (u, \ \mathrm{sym}_{n,k+1}) & \ge \eps, 
\end{align*}
then for any finite measure $\mu$ on $B_1$ we have 
\[
\beta_{\mu,k,2} (0,1)^2
\le C \int_{B_1} \left( \theta_u(y,8) - \theta_u(y,1) \right) \dd \mu(y).
\]
\end{theorem}
Again, the formulation in \cite{NabVal17} involves an energy bound. However, Theorem~\ref{th:uniformboundedness} shows a uniform bound on $\int_{B_9} |\nabla u|^2$ and thus we obtain the stronger formulation above. 

Since we shall only consider $k = n-3$, $p = 2$ and $\mu = \H^{n-3} \llcorner \sing u$ from now on, we abbreviate $\beta_{\mu,n-3,2}$ by $\beta$; this should not cause any confusion. 

\section{Uniform boundedness of Minimizers}

\begin{theorem}[Extension Property]\label{th:extensionthm}
Let $\Omega \subset \R^n$ be a bounded domain and let $v \in W^{1,2}(\Omega,\R^3)$ with $v(x)\in\S^2$ for a.e. $x\in\partial \Omega$. Then there exists a map $u \in W^{1,2}(\Omega,\S^2)$, 
\[
 u \Big\rvert_{\partial \Omega} =  v\Big |_{\partial \Omega}
\]
with the estimate 
\[
 \|\nabla u\|_{L^2(\Omega)}\leq C\,\|\nabla v\|_{L^2(\Omega)}
\]
for a uniform constant $C$.
\end{theorem}

That is, using trace theorems, Sobolev embedding, and Gagliardo-Nirenberg inequalities, we obtain as a corollary of Theorem~\ref{th:extensionthm} the following.
\begin{corollary}\label{co:est}

If $u: \B_r(0) \to \S^2$ is a minimizing harmonic map, then the following estimate holds
\begin{equation}\label{eq:traceextension}
 \|\nabla u\|_{L^2(B_r(0))} \aleq \sqrt{r^{\frac{n-1}{2}}\|\nabla_T u\|_{L^2(S_r)}}.
\end{equation}
If $u: \B_r^+(0) \to \S^2$ is a minimizing harmonic map with $u=\varphi$ on  the flat part of the boundary $T_r$. 
Then the following estimate hold
\begin{equation}\label{eq:traceestimatedonhalfball}
 \|\nabla u\|_{L^2(B_r^+(0))} \aleq \sqrt{r^{\frac{n-1}{2}}\|\nabla_T u\|_{L^2(S^+_r)} +  r^{\frac{n-1}{2}}\|\nabla_T \varphi\|_{L^{2}(T_r)}}.
\end{equation}
\end{corollary}

\begin{theorem}[Uniform Boundedness of Minimizers]\label{th:uniformboundedness}
Let $u\in W^{1,2}(B_R(0),\S^2)$ be a minimizing harmonic map. Then for any $r < R$,
\[
r^{2-n} \int_{B_r(0)} |\nabla u|^2 \dx\leq C\frac{R}{R-r},
\]
where $C$ is an absolute constant.

Also, let $u\in W^{1,2}(B_{2r}^+(0),\S^2)$ be a minimizing harmonic map. Then,
\[
 r^{2-n}\int_{B_{r}^+(0)} |\nabla u|^2 \dx \leq C \,\max \left \{ r^{\frac{3-n}{2}}\|\nabla_T u\|_{L^{2}(B_{2r}^+(0))} , 1 \right \}
\]
where $C$ is an absolute constant.
\end{theorem}
\begin{proof}
We prove the boundary estimate:

Denote by 
\[
D(\rho) := \|\nabla u\|_{L^2(B_\rho^+)}^2.
\]
Observe that then
\[
 D'(\rho) = \|\nabla u\|_{L^2(S_\rho^+)}^2.
\]
Then from Corollary~\ref{co:est} we have for any $\rho \in [r,2r]$,
\[
 D(\rho) \aleq \rho^{\frac{n-1}{2}} \sqrt{D'(\rho)} +  r^{\frac{n-1}{2}}\|\nabla_T u\|_{L^{2}(T_{2r})}
\]
or in other words
\[
\brac{C D(\rho) -  r^{\frac{n-1}{2}}\|\nabla_T u\|_{L^{2}(T_R)}} \aleq r^{\frac{n-1}{2}} \sqrt{D'(\rho)}.
\]
If it was the case that
\[
 r^{\frac{n-1}{2}}\|\nabla_T u\|_{L^{2}(B_{2r}^+)} \gg D(r)
 \]
we would end up with
\[
 D(r) \leq D(\rho) \aleq r^{\frac{n-1}{2}} \sqrt{D'(\rho)}     \quad \forall \rho \in [r,2r].
\]
That is,
\[
 r^{1-n}\aleq \frac{D'(\rho)}{(D(\rho))^2}     \quad \forall \rho \in [r,2r].
\]
Integrating on $(r,2r)$ we conclude
\[
 r^{2-n}\aleq \frac{1}{D(r)} - \frac{1}{D(2r)}    \quad \forall \rho \in [r,2r].
\]
That is,
\[
 D(r) \aleq r^{n-2}.
\]
We conclude that either 
\[
 r^{2-n} D(r) \aleq  r^{\frac{3-n}{2}}\|\nabla_T u\|_{L^{2}(B_{2r}^+)}
\]
or
\[
 r^{2-n} D(r) \aleq 1.
\]
That is,
\[
 r^{2-n} D(r) \aleq \max \left \{ r^{\frac{3-n}{2}}\|\nabla_T u\|_{L^{2}(B_{2r}^+)} , 1 \right \}.
 \]
\end{proof}

\subsection{Caccioppoli inequality and higher local integrability}

\begin{proposition}[Caccioppoli inequality (interior)]\label{pr:cacint}
Let $u\in W^{1,2}(\Omega,\S^{2})$ be a minimizing harmonic map. Suppose that $B_R(y)\subset\subset\Omega$ for some $y\in\Omega$ and $R>0$. Then there exists a constant C, such that
\begin{equation}\label{eq:comparisoncaccioppoli}
\int_{B_{R/2}(y)}|\nabla u|^2 \dx \leq C R^{-2} \int_{B_R(y)} |u-(u)_{B_{R}(y)}|^2 \dx.
\end{equation} 
Here $(u)_{B_R(y)}$ denotes the mean value.
\end{proposition}

\begin{proposition}[Caccioppoli inequality (boundary)]\label{pr:cacbdry}
Let $u\in W^{1,2}(B^+,\S^2)$ be a minimizing harmonic map and let $u=\vp$ on $T_1$ for a $\vp\in W^{\frac{1}{2},2}(T_1,\S^2)$. Then for all $r<1$ we have
 \[
  \int_{B^+_{r}}|\nabla u|^2 \dx \le C \int_{B_{2r}^+}|u(x) - \vp^h(x)|^2\dx + C\int_{B_{2r}^+}|\nabla \vp^h (x)|^2 \dx,
 \]
where $\vp^h \in W^{1,2}(\R^n_+,\R^3)$ is any harmonic function with $\varphi^h = \varphi$ on $T_1$.
\end{proposition}

As consequences of Poincar\'{e} inequality, Sobolev embedding and Gehring Lemma we readily obtain
\begin{corollary}[Higher integrability]\label{co:higherintint}
Let $u$ and $R$ be as in Proposition~\ref{pr:cacbdry}. There exists a $p>2$ such that 
\[
 \brac{\int_{B_{R/2}(y)}|\nabla u|^p \dx}^{1/p} \le C \brac{\int_{B_R(y)} |\nabla u|^2 \dx}^{1/2}.
\]
Let $u$ and $R$ be as in Proposition~\ref{pr:cacbdry}. 
\[
 \brac{R^{p-n} \int_{B^+_{R/2}}|\nabla u|^p \dx}^{1/p} \aleq  \brac{R^{2-n} \int_{B^+_R} |\nabla u|^2 \dx}^{1/2} + R^{\frac{3-n}{2}} \|\nabla_T u\|_{L^{2}(T_R)}.
\]
\end{corollary}

\section{Strong Convergence for minimizers and Consequences}

\begin{theorem}[strong convergence of minimizers]\label{th:bsc}
\hfill

\begin{enumerate}
\item Let $\{u_i\}_{i\in\N} \subset W^{1,2}(B,\S^2)$ be a sequence of minimizing harmonic maps. 
Then, up to taking a subsequence $i \to \infty$, we find $u \in W^{1,2}(\B, \S^2)$ which is a minimizer and $u_i\rightarrow u$ strongly in $W^{1,2}_{loc}(B,\S^2)$. 

\item Let $\{u_i\}_{i\in\N} \subset W^{1,2}(B^+,\S^2)$ be a sequence of minimizing harmonic maps and set $\vp_i := u_i$ on $T_1$. Additionally, assume that 
\[
 \sup_{i \in \N} [\varphi_i ]_{W^{1,2}(T_1)}  < \infty.
\]
Then, up to taking a subsequence $i \to \infty$, we find $u: \B^+ \to \S^2$ such that $u\in W^{1,2}(B_r^+,\S^2)$ for any $r \in (0,1)$ and $u_i\rightarrow u$ strongly in $W^{1,2}(B_{r}^+,\S^2)$. Moreover, for every $r \in (0,1)$ the map $u$ is a minimizing harmonic map in $B_r^+$.
\end{enumerate}
\end{theorem}
We will need the following lemma. 
\begin{lemma}[Poincar\'e-type Lemma]\label{la:1dpoinc}
Let $f \in W^{1,2}(B^+_1)$ be such that in the sense of traces $f = 0$ on $T_{3/4}$. Then, for any $\delta \in (0,\frac{1}{2})$,
\[
 \int_{T_{3/4} \times (0,\delta)} |f|^2 \aleq \delta^2 \int_{T_{3/4} \times (0,\delta)} |\nabla f|^2.
\]
\end{lemma}
\begin{proof}
For continuous functions $\varphi: [0,\delta] \to \R$ we have by the fundamental theorem of calculus,
\[
 \int_0^\delta |\varphi(0) - \varphi(t)|^2\, dt \leq \delta^2 \int_0^\delta |\varphi'(t)|^2\, dt.
\]
For almost all $x' \in T_{3/4}$ we thus have
\[
 \int_0^\delta |f(x',t)|^2 dt = \int_0^\delta |f(x',0) - f(x',t)|^2 dt \leq \delta^2 \int_0^\delta |\nabla f(x',t)|^2\, dt.
\]
Integrating this in $T_{3/4}$ we obtain the claim.
\end{proof}

\begin{proof}[Proof of Theorem~\ref{th:bsc}]
We will give only the proof of (2). The interior convergence follows similarly, see for example \cite[Theorem 6.4]{HL1987} (see also \cite[Theorem 4.6]{MMS18} for $n=3$).

From Theorem~\ref{th:uniformboundedness} we have
\[
 \sup_{i \in \N} [u_i]_{W^{1,2}(B^+_r)} < \infty \quad \mbox{for any $r \in (0,1)$}.
\]
In particular, up to taking a sequence and diagonalizing we find $u : \B^+ \to \S^2$ which is a weak $W^{1,2}$-limit, and strong $L^2$-limit of $u_i$ in each ball $B^+_r$, and $\varphi$ as the weak $W^{1,2}$-limit of $\varphi_i$ on each $T_r$, such that $\varphi$ is the trace of $u$.

We need to show that $u$ is a minimizer in $B^+_r$ and that $u_k \to u$ strongly with respect to the $W^{1,2}$-norm in $B^+_r$ for every $r \in (0,1)$. For simplicity of notation we shall assume $r = \frac{1}{2}$.

By Corollary~\ref{co:higherintint} we have uniformly higher integrability of $v_i$, namely for some $p > 2$ we have
\begin{equation}\label{eq:bsc:higherintegr}
\sup_{i} \int_{B^+_{3/4}} |\nabla u_i|^p < \infty.
\end{equation}
Now let $v\in W^{1,2}(B_{1/2}^+,\S^2)$ be a map that coincides with $u$ on $\partial B_{1/2}^+$, namely such that $v = \vp$ on $T_1$ and $v\big\rvert_{S_{1/2}^+} = u \big\rvert_{S_{1/2}^+}$. We extend $v$ by $u$ to all of $B_{3/4}^+$ and thus find $v \in W^{1,2}(B_{3/4}^+,\S^2)$, $v \equiv u$ on $B_{3/4}^+ \backslash B_{1/2}^+$.

\begin{center}
\includegraphics[width=300px]{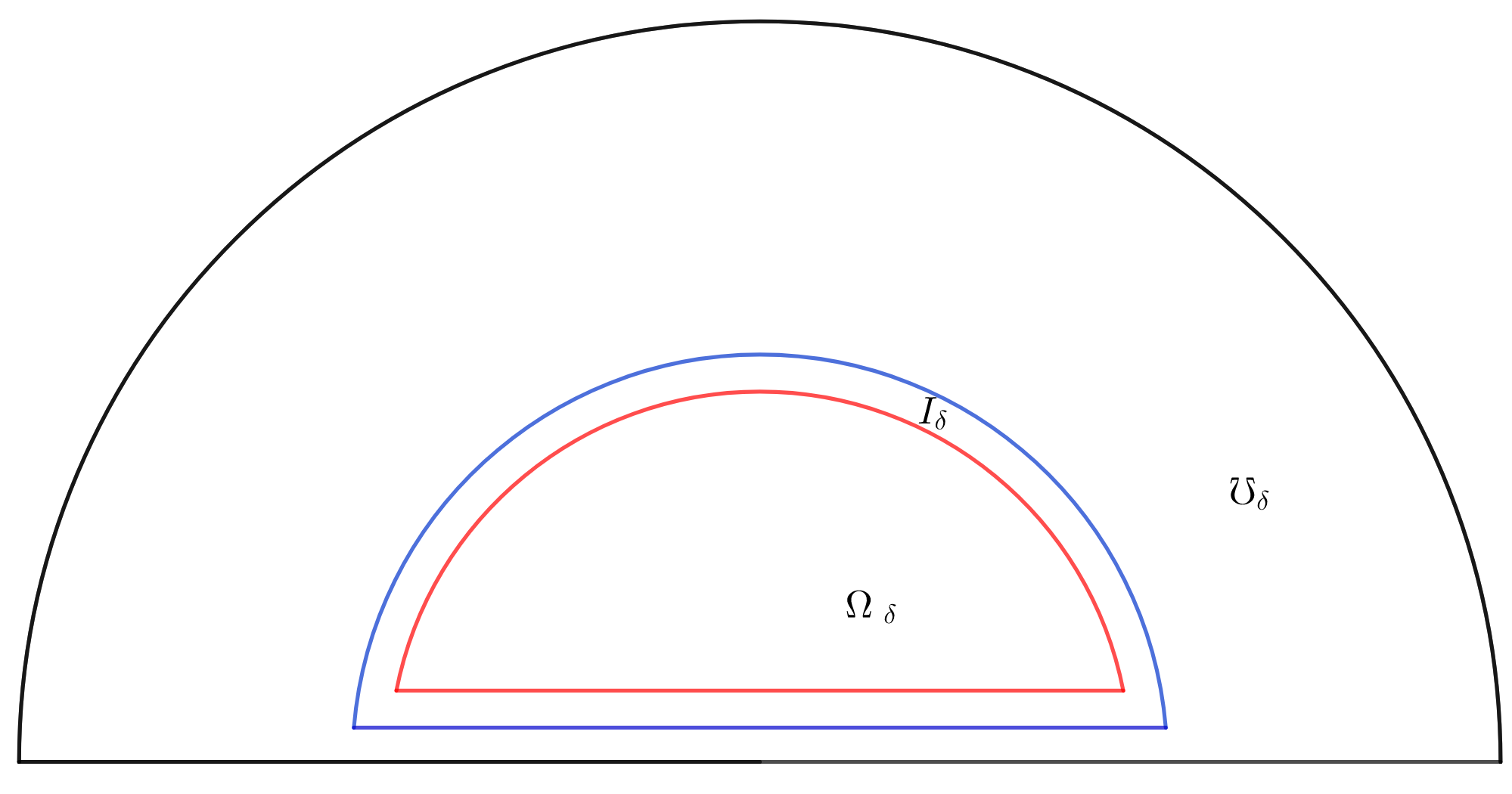}
\end{center}

The map $v$ is a competitor for $u$, and we need to transform it into a competitor for $u_i$. We do so by an interpolation on a set $I_\delta$ which separates $\Omega_\delta$ and $\mho_\delta$, which are defined as follows:
\[
 \Omega_\delta := B_{1/2}^+ \backslash \brac{\R^{n-1} \times (0,2\delta)};
\]
\[
 I_\delta = B^+_{1/2+\delta} \backslash \brac{\Omega_\delta \cup \brac{\R^{n-1} \times (0,\delta)} };
\]
\[
 \mho_\delta := B_{3/4}^+ \backslash \brac{I_\delta \cup \Omega_\delta}.
\]
Let $\eta_\delta \in C_c^\infty(\Omega_\delta \cup I_\delta)$ be a cutoff function, $\eta \in [0,1]$, $\eta_\delta \equiv 1$ in $\Omega_\delta$, with $|\nabla \eta_\delta| \aleq \frac{1}{\delta}$.

We set
\begin{equation}\label{eq:vdeltait}
 \tilde{v}_{\delta,i} := \eta_\delta v + (1-\eta_\delta) u_i = u_i + \eta_\delta (v-u_i) \quad \mbox{in $I_\delta$}.
\end{equation}
Observe that on $\partial I_\delta$ is separated into two parts, the inner part being $\partial \Omega_\delta$ and the outer being $\partial (I_\delta \cup \Omega_\delta)$.

We have $\tilde{v}_{\delta,i} = v$ on $\partial \Omega_\delta$, and $\tilde{v}_{\delta,i} = u_i$ on $\partial (I_\delta \cup \Omega_\delta)$. However, $\tilde{v}_{\delta,i}$ does not map into the sphere. So we use the extension theorem, Theorem~\ref{th:extensionthm}, and obtain some $v_{\delta,i}: I_\delta \to \S^2$ with the same boundary data, i.e., 
$v_{\delta,i} = v$ on $\partial \Omega_\delta$, and $v_{\delta,i} = u_i$ on $\partial (I_\delta \cup \Omega_\delta)$. Moreover, we have 
\begin{equation}\label{eq:nv1}
 \int_{I_\delta} |\nabla v_{\delta,i}|^2 \aleq \int_{I_\delta} |\nabla \tilde{v}_{\delta,i}|^2,
\end{equation}
with the constant independent of $i$ and $\delta$.

We extend $v_{\delta,i}$ by $u_i$ to $\mho_\delta$ and by $v$ to $\Omega_\delta$, and thus have $v_{\delta,i} \in W^{1,2}(B_{3/4}^+,\S^2)$ such that 
\[
 v_{\delta,i} = \begin{cases}
                 v \quad \mbox{in $\Omega_\delta$},\\
                 u_i \quad \mbox{in $\mho_\delta$}.
                \end{cases}
\]
In particular, $v_{\delta,i}$ is a competitor for $u_i$ on $B^+_{3/4}$, and the minimizing property of $u_i$ implies
\begin{equation}\label{eq:sc:1}
\begin{split}
 \int_{B_{3/4}^+} |\nabla u_i|^2 \leq& \int_{B_{3/4}^+} |\nabla v_{i,\delta}|^2\\
 =&\int_{\mho_\delta} |\nabla u_i|^2 + \int_{\Omega_\delta} |\nabla v|^2 + 
 \int_{I_\delta} |\nabla v_{i,\delta}|^2.
\end{split}
 \end{equation}
We observe that
\[
 \int_{\mho_\delta} |\nabla u_i|^2 \leq \int_{B_{3/4}^+ \backslash B_{1/2}^+} |\nabla u_i|^2 + \int_{T_{1/2} \times (0,\delta)} |\nabla u_i|^2
\]
and
\[
 \int_{\Omega_\delta} |\nabla v|^2 \leq \int_{B_{1/2}^+} |\nabla v|^2.
\]
Thus, \eqref{eq:sc:1} becomes
\begin{equation}\label{eq:sc:2}
\begin{split}
 \int_{B_{1/2}^+} |\nabla u_i|^2 \leq \int_{T_{1/2}\times(0,\delta)} |\nabla u_i|^2 + \int_{B_{1/2}^+} |\nabla v|^2 + 
 \int_{I_\delta} |\nabla v_{i,\delta}|^2.
\end{split}
\end{equation}
Moreover, by \eqref{eq:nv1} and \eqref{eq:vdeltait}
\[
\begin{split}
 \int_{I_\delta} |\nabla v_{i,\delta}|^2  \leq& C\int_{I_\delta} |\nabla \tilde{v}_{i,\delta}|^2\\
 \aleq & \int_{I_\delta} |\nabla u_i|^2 + \int_{I_\delta} |\nabla v|^2 + \frac{1}{\delta^2} \int_{I_\delta} |u_i-v|^2\\
 \aleq & \int_{I_\delta} |\nabla u_i|^2 + \int_{I_\delta} |\nabla v|^2 + \frac{1}{\delta^2} \int_{I_\delta} |u-v|^2+\frac{1}{\delta^2} \int_{I_\delta} |u_i-u|^2.
\end{split}
 \]
Also, observe that $u = v$ in $B_{3/4}^+\backslash B_{1/2}^+$. That is,
\[
 \frac{1}{\delta^2} \int_{I_\delta} |u-v|^2 \leq \frac{1}{\delta^2} \int_{T_{3/4} \times (0,2\delta)} |u-v|^2 
\]
Since moreover $u = v$ on $T_{3/4}$, we can apply Lemma~\ref{la:1dpoinc}, and obtain
\[
 \frac{1}{\delta^2} \int_{T_{3/4} \times (0,2\delta)} |u-v|^2 \aleq \int_{T_{3/4} \times (0,2\delta)} |\nabla (u-v)|^2.
\]
We thus arrive at
\begin{equation}\label{eq:sc:3}
\begin{split}
 \int_{B_{1/2}^+} |\nabla u_i|^2 \leq& \int_{B_{1/2}^+} |\nabla v|^2 + \int_{T_{1/2}\times(0,\delta)} |\nabla u_i|^2 \\
 &+  C \int_{I_\delta} |\nabla u_i|^2 + C\int_{I_\delta} |\nabla v|^2 + \int_{T_{3/4} \times (0,2\delta)} |\nabla (u-v)|^2+\frac{C}{\delta^2} \int_{I_\delta} |u_i-u|^2.
\end{split}
\end{equation}
This estimate holds for all $i$ and $\delta$, so taking the limit superior $i \to \infty$, by Fatou's lemma for weakly convergent $u_i$ to $u$,
\begin{equation}\label{eq:sc:4}
\begin{split}
 \limsup_{i \to \infty} \int_{B_{1/2}^+} |\nabla u_i|^2 \leq& \int_{B_{1/2}^+} |\nabla v|^2 + \sup_{i}\int_{T_{1/2}\times(0,\delta)} |\nabla u_i|^2 \\
 &+   \sup_{i}C \int_{I_\delta} |\nabla u_i|^2 +  C\int_{I_\delta} |\nabla v|^2\\
 &+ C\int_{T_{3/4} \times (0,2\delta)} |\nabla (u-v)|^2+\frac{C}{\delta^2}  \limsup_{i \to \infty}\int_{I_\delta} |u_i-u|^2.
\end{split}
\end{equation}
Since $u_i$ converges strongly in $L^2$ to $u$ on $I_\delta \subset B_{3/4}^+$, we have
\[
 \frac{C}{\delta^2}  \limsup_{i \to \infty}\int_{I_\delta} |u_i-u|^2 = 0.
\]
That is, we have shown
\begin{equation}\label{eq:sc:5a}
\begin{split}
 \limsup_{i \to \infty} \int_{B_{1/2}^+} |\nabla u_i|^2 \leq& \int_{B_{1/2}^+} |\nabla v|^2 + \sup_{i}\int_{T_{1/2}\times(0,\delta)} |\nabla u_i|^2 \\
 &+   \sup_{i}C \int_{I_\delta} |\nabla u_i|^2 +  C\int_{I_\delta} |\nabla v|^2\\
 &+ C\int_{T_{3/4} \times (0,2\delta)} |\nabla (u-v)|^2.
\end{split}
\end{equation}
This holds for any $\delta > 0$ small enough, so we take the limit $\delta \to 0$ and obtain
\begin{equation}\label{eq:sc:5}
\begin{split}
 \limsup_{i \to \infty} \int_{B_{1/2}^+} |\nabla u_i|^2 \leq& \int_{B_{1/2}^+} |\nabla v|^2 + \lim_{\delta \to 0} \sup_{i}\int_{T_{1/2}\times(0,\delta)} |\nabla u_i|^2 \\
 &+   \lim_{\delta \to 0}\sup_{i}C \int_{I_\delta} |\nabla u_i|^2 +  \lim_{\delta \to 0} C\int_{I_\delta} |\nabla v|^2\\
 &+ \lim_{\delta \to 0} C \int_{T_{3/4} \times (0,2\delta)} |\nabla (u-v)|^2.
\end{split}
\end{equation}
By absolute continuity of the integral we observe
\[
 \lim_{\delta \to 0} \int_{I_\delta} |\nabla v|^2 = \lim_{\delta \to 0} \int_{T_{3/4} \times (0,2\delta)} |\nabla (u-v)|^2 =  0.
\]
Moreover, by H\"older inequality and the higher integrability of $u_i$, \eqref{eq:bsc:higherintegr}, we find for some $p > 2$,
\[
 \sup_{i}\int_{T_{1/2}\times(0,\delta)} |\nabla u_i|^2  +   \sup_{i}C \int_{I_\delta} |\nabla u_i|^2
 \aleq \left |T_{1/2} \times (0,\delta) \right |^{1-\frac{2}{p}} + |I_\delta|^{1-\frac{2}{p}} \xrightarrow{\delta \to 0} 0.
\]
That is, finally, by Fatou's lemma and weak convergence of $u_i$ to $u$ this implies readily
\begin{equation}\label{eq:sc:HUNGRY}
 \int_{B_{1/2}^+} |\nabla u|^2  \leq \limsup_{i \to \infty} \int_{B_{1/2}^+} |\nabla u_i|^2 \leq \int_{B_{1/2}^+} |\nabla v|^2.
\end{equation}
This holds for all $v$ which coincide with $u$ on $\partial B_{1/2}^+$, in particular we have shown that $u$ is a minimizing harmonic map in $B_{1/2}^+$.

On the other hand, from \eqref{eq:sc:HUNGRY} we get by taking $v \equiv u$,
\[
 \int_{B_{1/2}^+} |\nabla u|^2 \leq \limsup_{i \to \infty} \int_{B_{1/2}^+} |\nabla u_i|^2\leq \int_{B_{1/2}^+} |\nabla u|^2.
\]
That is,
\[
 \lim_{i \to \infty} \int_{B_{1/2}^+} |\nabla u_i|^2 = \int_{B_{1/2}^+} |\nabla u|^2,
\]
and thus, using once again the weak convergence of $u_i$ to $u$,
\[
 \lim_{i \to \infty} \int_{B_{1/2}^+} |\nabla u_i-\nabla u|^2 = \lim_{i \to \infty} \brac{ \int_{B_{1/2}^+} |\nabla u_i|^2+|\nabla u|^2- 2 \nabla u_i \cdot \nabla u} = 0.
\]
That is $u_i \to u$ in $W^{1,2}(B_{1/2}^+)$.
\end{proof}

\begin{remark}
\label{rem:non-flat}
A technical modification of this reasoning allows us to consider in Theorem~\ref{th:bsc} a sequence of maps $u_i$ defined on converging Lipschitz domains with non-flat boundaries. This will be used in Theorem~\ref{th:hot-spots}. 
\end{remark}

\subsection{Smoothness for small boundary data}
As a first corollary of the compactness results above, Theorem~\ref{th:bsc}, we have 

\begin{theorem}[interior regularity for almost constant boundary data]
\label{th:int-regularity-in-terms-of-bdry}
For each bounded smooth domain $\Omega \subset \R^n$, there are small constants $\sigma(\Omega,n) > 0$ and $\eps(n) > 0$ so that the following holds. If $u \in W^{1,2}(\Omega,\S^2)$ is a minimizing harmonic map with trace $\varphi := u \Big |_{\partial \Omega}$ and 
\[
\int_{\partial \Omega} |\nabla \vp|^{n-1} \dd \cH^{n-1}\le \eps,
\]
then $u$ is smooth in the interior region $\{ x \in \Omega : \dist(x, \partial \Omega) > \sigma \}$. 
\end{theorem}

\subsection{Convergence of singular points}

\begin{theorem}[{Singular points converge to singular points, {\cite[Thm 1.8]{AlmgrenLieb1988}}} ]\label{th:ALs2s}
Assume that a sequence of energy minimizing maps $u_k \in W^{1,2}(\Omega, \S^2)$ converges strongly in $W^{1,2}_{loc}$ to $u$, and a sequence of their singularities $y_k \in \sing u_k$ converges to $y \in \Omega$. Then $y$ is a singular point of $u$. 
\end{theorem}
\begin{proof}
 The proof is essentially the same as in \cite[Theorem 4.8 (i)]{MMS18}.
\end{proof}

\section{Boundary regularity for smooth and singular boundary data in \texorpdfstring{$W^{1,n-1}$}{W(1,n-1)}}

\subsection{Uniform boundary regularity for constant boundary data}
The first step is uniform boundary regularity for constant boundary data, see \cite[Theorem 1.10]{AlmgrenLieb1988}.
\begin{theorem}[Boundary regularity]\label{th:boundaryregularity-constant}
There exists a uniform constant $\lambda > 0$ such that the following holds:
Let $u\in W^{1,2}(\B^+_1,\S^2)$ be a minimizer and assume that $\varphi = u\Big |_{T_1}$ is constant.
Then $u$ is analytic in
\[
 [0,\lambda] \times T_{1/2}.
 \]
\end{theorem}

The main ingredient in Theorem~\ref{th:boundaryregularity-constant} is the following. Here we state it without the proof as, after adjusting the dimension, it follows from the proof in \cite[Lemma 5.2]{MMS18}.
\begin{lemma}\label{la:uniformsmallnessboundaryconstant}
For any $\eps > 0$ there is a uniform constant $R_0(\eps) \in (0,\frac{1}{2})$ so that the following holds:
Let $u\in W^{1,2}(\B^+,\S^2)$ be a minimizer and assume that $\varphi = u\Big |_{T_1}$ is a constant.
Then for any $x_0 \in T_{1/2}$
\[
\sup_{r < R_0(\eps)} r^{2-n} \int_{B_r(x_0)} |\nabla u|^2 \dx < \eps.
 \]
\end{lemma}

\begin{proof}[Proof of Theorem~\ref{th:boundaryregularity-constant}]
The proof is essentially the same as \cite[Theorem 5.1]{MMS18}, for analyticity of the solutions we refer to \cite{BG80,T72}.
\end{proof}
 
\subsection{Uniform boundary regularity for singular boundary data}

\begin{theorem}[Uniform boundary regularity for singular boundary data]\label{th:r4s}
Let $\Omega \subset \R^n$ be a bounded domain with smooth boundary. There are constants $R = R(\Omega)$ and $\eps = \eps(\Omega)$ such that the following holds. 

Take any minimizing harmonic map $u:\Omega \to \S^2$ and denote the trace of $u$ on $\partial \Omega$ by $\varphi$.

If for some $x_0 \in \partial \Omega$ and some $\rho < R$ we have the estimate
\[
\Lambda := \int_{T_\rho(x_0)} |\nabla_T \varphi |^{n-1}  d\mathcal{H}^{n-1}\leq \eps
\]
then $u$ is smooth in $B_{x_0}(\rho/2) \cap \Omega$. 
\end{theorem}

\begin{proposition}\label{th:rs:flat}
There exist uniform constants $R_0$ and $\eps$ such that the following holds.
Take any minimizing harmonic map $u:B^+ \to \S^2$ and denote the trace of $u$ on $T_1$ by $\varphi$.

If for some $\rho_0 < R_0$ we have the estimate
\[
[\varphi ]_{W^{1,n-1}(T_1 \cap B_{\rho_0})} \leq \eps
\]
then $u$ is smooth in 
\[
B_{\lambda \rho_0} \cap \{x_3 > \lambda\, \rho_0 /2\}
\]
where $\lambda$ is from Theorem~\ref{th:boundaryregularity-constant}.
\end{proposition}
\begin{proof}
The proof is essentially the same as in \cite[Proposition 5.4]{MMS18}.
\end{proof}

By a covering argument, we obtain in particular the following regularity up to the boundary (but of course not including the boundary).
\begin{corollary}\label{co:flatreg}
There exist uniform constants $R_0$ and $\eps$ such that the following holds.
Take any minimizing harmonic map $u:B^+_2 \to \S^2$ and denote the trace of $u$ on $T_2$ by $\varphi$.

If for some $x_0\in T_1$ and some $\rho < R_0$ we have the estimate
\[
 [\varphi ]^{n-1}_{W^{1,n-1}(T_1 \cap B_{\rho}(x_0))} \leq \eps
\]
then $u$ is smooth in $B_{\frac{\lambda}{2} \rho}(x_0)\cap B^+_2$, where $\lambda$ is as in Theorem~\ref{th:boundaryregularity-constant}.
\end{corollary}

\begin{proof}[Proof of Theorem~\ref{th:r4s}]
The proof of Theorem~\ref{th:r4s} follows now from Corollary~\ref{co:flatreg} by a blowup argument.
\end{proof}

\subsection{Hot spots}
\begin{theorem}[%
{\color{red} R}%
{\color{orange}a}%
{\color{yellow}i}%
{\color{green}n}%
{\color{blue}b}%
{\color{indigo}o}%
{\color{violet}w} theorem]\label{th:AL23}
Let $\Omega \subset \R^n$ be a bounded domain with smooth boundary. There exists a number $r_0 = r_0(\Omega)>0$, with the following property. 

For $x_0 \in \partial\Omega$ let $A_{(r,s)}(x_0):= \{x\in\R^n\colon r<\dist(x,x_0)<s\} $. Suppose also that $u$ is a minimizer in $\Omega$ having boundary map $\vp$. Then, whenever $0<r<r_0$,
\begin{equation}
 r^{2-n}\int_{\Omega \cap A_{(r,2r)}(x_0)}|\nabla u|^2 \dx \le C + Cr^{3-n}\int_{\partial \Omega \cap  A_{ (r/2, 5r/2)}(x_0)} |\nabla_T \varphi|^2 \dhn,
\end{equation}
where $C$ is a constant independent of $\Omega, K, u,$ and $\vp$.
\end{theorem}
\begin{proof}
We can cover the annulus $A_{(r,2r)}(x_0)$ by balls, of radius comparable to $r$ not leaving $A(r/2,5r/2)$. The number of these balls is a dimensional constant, and we get the desired estimate by the uniform boundedness theorem, Theorem~\ref{th:uniformboundedness}. See \cite[Theorem 6.1]{MMS18}.
\end{proof}

\begin{theorem}[regularity away from ``hot spots'']\label{th:AL24}
For every $N \in \N$ there exists an $\eps_{N}>0$ with the following property.  Suppose $u\in W^{1,2}(B^+_1,\S^2)$ with trace $\varphi$ on $T_1$.
Assume that there are balls $\mathcal B_1,\ldots,\mathcal B_N$ of radius at most $\eps_N$ such that
\[
 \int_{T_{1} \setminus \brac{\mathcal B_1\cup\ldots \mathcal B_N}}|\nabla \varphi|^2\, d\mathcal{H}^{n-1}< \eps_N.
\]
Then $u$ is smooth in
\[
 T_{1/2} \times (\mu,2\mu),
\]
for a uniform constant $\mu > 0$.
\end{theorem}

\begin{proof}
For simplicity suppose that $N=1$, $\eps_N = \eps_1=\eps$. Thus we have only one ball $\mathcal B_1 = B_{\eps}(p)$. 

We argue by contradiction. Assume that $u_i \colon B^+_1 \to \S^2$ is a sequence of minimizers with boundary maps $\vp_i$ such that
\[
\int_{T_{1} \backslash B_{\eps_i}(p_i)} |\nabla \varphi_i|^2 \dhn< \eps_i
\] 
for a sequence of balls $B_{\eps_i}(p_i)$ and $\eps_i \xrightarrow{i \to \infty} 0$. Setting $r_i := (\eps_i)^{\frac{1}{n-2}} > \eps_i$ we in particular obtain
\begin{equation}\label{eq:varphitrace}
\brac{r_i}^{3-n} \int_{T_{1} \backslash B_{r_i}(p_i)} |\nabla \varphi_i|^2 \dhn< r_i,
\end{equation}
where $r_i \xrightarrow{i \to \infty} 0$, and up to taking a subsequence, $r_i < 2^{-i}$.

Now, we assume (by contradiction) that each $u_i$ has at least one point singularity $y_i \in T_{1/2} \times (\mu,2\mu)$.

By Theorem~\ref{th:AL23}, for large enough $i$ and for any $r \ge 2^{-i}$
\[
r^{2-n}\int_{B^+_1 \cap A_{(r,2r)}(p_i)}|\nabla u_i|^2 \dx \le C.
\]
Thus, for every $k \geq i$,
\[
\int_{B^+_1 \cap A_{(2^{k},2^{k+1})}(p_i)}|\nabla u_i|^2 \dx \le C\, 2^{k(n-2)}.
\]
Up to taking another subsequence we can assume that $p_i \to p_0$, and that $|p_i-p_0| \leq 2^{-i}$. Then, from the above estimate we have
\[
\int_{B^+_{4/5} \setminus B_{2^{i+13}}(p_0)}|\nabla u_i|^2 \dx \le C\, \sum_{k=-i}^{0} 2^{k(n-2)} \leq C.
\]
In particular by a diagonal argument and the strong convergence of minimizers, Theorem~\ref{th:bsc}, we obtain a minimizer $u \in W^{1,2}(B_{3/4}^+ \backslash B_{r}(p_0))$ for any $r > 0$. Moreover, its trace, which we shall call $\varphi \in W^{1,2}(T_{1} \backslash B_r(p_0),\S^2)$ is the weak limit of $\varphi_i$. Observe that $\varphi$ is constant on $T_{1} \backslash B_r(p_0)$ for any $r > 0$, by \eqref{eq:varphitrace}.

Moreover, by Theorem~\ref{th:ALs2s}, the sequence of singular points $y_i$ can be assumed to converge to a singular point of $u$ which we call $y \in T_{1/2} \times (\mu,2\mu)$.

On the other hand, in view of Lemma~\ref{lem:removesg} below, the singularity $p_0$ is removable, and so $u$ is a minimizing harmonic map in $B_{3/4}^+$. Since $u$ is constant on $T_{3/4}$, if we choose $\mu := \frac{1}{2} \lambda$ where $\lambda$ is from Theorem~\ref{th:boundaryregularity-constant} we find a contradiction.
\end{proof}

To complete the proof of Theorem~\ref{th:AL24}, we need the following removability lemma. 

\begin{lemma}[Removability of points for minimizing harmonic maps]
\label{lem:removesg}
Assume that $u \in W^{1,2}(B_1^+(0)\backslash B_\delta(0),\S^2)$ for any $\delta > 0$ with
\[
\sup_{\delta \in (0,1)} \int_{B^+_1(0) \backslash B_\delta(0)} |\nabla u|^2 \dx < \infty
\]
is a minimizer away from the origin, i.e., assume that for any $\delta > 0$ and any $v \in W^{1,2}(B_1^+(0),\S^2)$ satisfying $v = u$ on $\partial B_1^+(0)$ and $v = u$ on $\partial B_\delta^+(0)$ we have 
\begin{equation}\label{eq:dmin}
\int_{B_1^+(0) \backslash B_{\delta}^+(0)} |\nabla u|^2 \dx \leq \int_{B_1^+(0) \backslash B_{\delta}^+(0)} |\nabla v|^2 \dx.
\end{equation}
Then, $u \in W^{1,2}(B^+_1(0))$ and $u$ a minimizing harmonic map in all of $B_1^+(0)$.
\end{lemma}

\begin{proof}
Since $u$ and $\partial_i u$ are uniformly in $L^2$ on $B^+_1 \backslash B^+_\delta$ there exists extensions to all of $B^+_1$, which satisfy
\[
\int_{B^+_1} u\, \partial_i \varphi = -\int_{B^+_1} \partial_i u\, \varphi\quad \mbox{for all $\varphi \in C_c^\infty(B^+_1)$ such that $\varphi(x) = 0$ around $x = 0$}.
\]
To show that this also holds for generic $\varphi \in C_c^\infty(B^+_1)$ let $\eta_\delta \in C_c^\infty(B_{2\delta})$ be a cutoff function constantly one around $B_\delta$ with $|\nabla \eta_\delta| \aleq \frac{1}{\delta}$. Then
\[
\begin{split}
 \int_{B^+_1} u\, \partial_i \varphi  =& \int_{B^+_1} u\, \partial_i ((1-\eta_\delta) \varphi ) + \int_{B^+_1} u\, \partial_i (\eta_\delta \varphi )\\
 =& -\int_{B^+_1} \partial_i u\, ((1-\eta_\delta) \varphi ) + \int_{B^+_1} u\, \partial_i (\eta_\delta \varphi )\\
 =& -\int_{B^+_1} \partial_i u\,  \varphi  +\int_{B^+_1} \partial_i u\, \eta_\delta \varphi  + \int_{B^+_1} u\, \partial_i (\eta_\delta \varphi ).\\
\end{split}
 \]
This holds for any $\delta \in (0,1)$ and we observe that by absolute continuity of the integral
\[
 \lim_{\delta \to 0} \int_{B^+_1} \partial_i u\, \eta_\delta \varphi  \aleq \lim_{\delta \to 0} \int_{B^+_{2\delta}} |\partial_i u|= 0
 \]
 and, since $n \geq 3$,
 \[
  \lim_{\delta \to 0} \int_{B^+_1} u\, \partial_i (\eta_\delta \varphi ) \aleq \lim_{\delta \to 0} (1+\frac{1}{\delta}) \int_{B^+_{2\delta}} |u|
   \aleq \lim_{\delta \to 0} (1+\frac{1}{\delta}) \delta^{\frac{n}{2}} \|u\|_{L^2(B^+_{2\delta})} = 0.
 \]
Thus, $u \in W^{1,2}(B^+_1)$.

Now let $w \in W^{1,2}(B_1^+,\S^2)$ with $u \equiv w$ on $\partial B_1^+$ be a competitor. We need to show that 
\begin{equation}\label{eq:min}
\int_{B_1^+} |\nabla u|^2 \dx \leq \int_{B_1^+} |\nabla w|^2 \dx.
\end{equation}
For $\delta > 0$ let $\eta_\delta \in C_c^\infty(B_{2\delta})$ be again the typical cutoff function, $\eta_\delta \equiv 1$ in $B_{\delta}$ and $|\nabla \eta_\delta| \aleq \frac{1}{\delta}$.

We set $\tilde{w}_\delta\in W^{1,2}(B_1^+,\R^3)$ by
\[
\tilde{w}_\delta := (1-\eta_\delta) w + \eta_\delta u,
\]
which satisfies $\tilde{w}_\delta = u$ on $\partial B_1^+$, $\tilde{w}_\delta \equiv u$ in $B_\delta^+$ and $\tilde{w}_\delta \equiv w$ in $B_1^+ \backslash B_{2\delta}$. By the extension property, Theorem~\ref{th:extensionthm}, applied in $B_{2\delta}^+ \backslash B_\delta$ we find $w_\delta \in W^{1,2}(B_1^+,\S^2)$ such that
\[
w_\delta = \begin{cases}
u \quad \mbox{in $B_\delta^+$}\\
w \quad \mbox{in $B_1^+ \backslash B_{2\delta}$}\\
u \quad \mbox{on $\partial B_1$}
\end{cases}
\]
and
\[
\int_{B_{2\delta}^+ \backslash B_\delta} |\nabla w_\delta|^2 \dx \aleq 
\int_{B_{2\delta}^+ \backslash B_\delta} |\nabla \tilde{w}_\delta|^2 \dx.
\]
In particular, $\tilde{w}_\delta$ is a competitor in the sense of \eqref{eq:dmin}, and we have 
\[
\begin{split}
\int_{B_1^+ \setminus B_{\delta}} |\nabla u|^2 \dx &\leq \int_{B_1^+ \backslash B_{\delta}} |\nabla w_\delta|^2 \dx\\
 &= \int_{B_1^+ \backslash B_{2\delta}} |\nabla w_\delta|^2 \dx + \int_{B_{2\delta}^+ \backslash B_{\delta}} |\nabla w_\delta|^2\dx \\
&\leq  \int_{B_1^+ \backslash B_{2\delta}} |\nabla w|^2 \dx + C\int_{B_{2\delta}^+} |\nabla \tilde{w}_\delta|^2 \dx.\\
\end{split}
\]
Since $u$, and $w \in W^{1,2}(\B_1^+)$ using the absolute continuity of the integral we find that
\begin{equation}\label{eq:almminlim}
\int_{B_1^+} |\nabla u|^2 \dx \leq \int_{B_1^+} |\nabla w|^2 \dx + C\, \liminf_{\delta \to 0}\int_{B_{2\delta}^+} |\nabla \tilde{w}_\delta|^2 \dx.
\end{equation}
Now
\[
\int_{B_{2\delta}^+} |\nabla \tilde{w}_\delta|^2 \dx \aleq \frac{1}{\delta^2} \int_{B_{2\delta}^+} |u-v|^2 \dx + \int_{B_{2\delta}^+} |\nabla u|^2 \dx + \int_{B_{2\delta}^+} |\nabla v|^2 \dx.
\]
Observe that we are in dimension $n \geq 3$ and $\S^2$ is compact, so
\[
\frac{1}{\delta^2} \int_{B_{2\delta}^+} |u-v|^2 \dx\aleq \delta.
\]
Thus, using again the absolute continuity of the integral and that $u,w \in W^{1,2}$ we find
\[
\lim_{\delta \to 0} \int_{B_{2\delta}^+} |\nabla \tilde{w}_\delta|^2 \dx = 0.
\]
Plugging this into \eqref{eq:almminlim} we conclude.
\end{proof}

As a corollary of Theorem~\ref{th:AL24} we obtain
\begin{theorem}[boundary regularity with hot spots]
\label{th:hot-spots}
For each bounded smooth domain $\Omega \subset \R^n$, there are small constants $\sigma, \eps, \lambda > 0$, $\Lambda > 1$, ($\sigma$ depending on the geometry of $\Omega$, the others only on the dimension) so that the following statement holds true for any minimizer $u \in W^{1,2}(\Omega,\S^2)$ with trace $\varphi := u \Big |_{\partial \Omega}$.

For any singular point $p \in \sing u$ with $r := \dist(p,\partial \Omega) < \sigma$ and for any ball $B \subset \R^n$ with radius $\lambda r$, we have
\[
r^{3-n}\int_{\partial \Omega \cap \brac{B_{\Lambda r}(p) \setminus B}} |\nabla \vp|^{2} \dd \cH^{n-1}  \geq \eps. 
\]
\end{theorem}
\begin{proof}
In principle, this is a rescaled version of Theorem~\ref{th:AL24}, only with non-flat boundary. By choosing $\sigma > 0$ small enough, we can ensure that after rescaling the balls to unit size, the boundary is arbitrarily close to flat. To consider this more general case, one needs another contradiction argument based on Theorem~\ref{th:bsc} (see Remark \ref{rem:non-flat}). 
\end{proof}

\section{Hardt and Lin's stability of singularities for \texorpdfstring{$n \geq 3$}{n >= 3}}
This chapter is concerned with stability of singularities. By this we mean that if two boundary maps $\vp, \vp' \colon \partial \Omega \to \S^2$ are \emph{close} in the right Sobolev norm, then the singularities of their corresponding minimizers $u, u' \colon \Omega \to \S^2$ are \emph{close} as well. Since minimizers are in general non-unique, the precise statement is a little more subtle -- e.g. by assuming uniqueness a priori. 

In any case, let us discuss the right notions of \emph{closeness}. In dimension $n = 3$, when the singular set consists of finitely many points, Hardt and Lin \cite{HL1989} considered the Lipschitz norm for boundary data, and showed that small perturbations do not change the number of singularities. Moreover, they constructed a bi-Lipschitz diffeomorphism $\eta \colon \Omega \to \Omega$ (close to identity in Lipschitz norm) such that $u$ is close to $u' \circ \eta$ in some $C^\beta$ norm. These results were recently extended to the case of $W^{1,2}$-perturbations of boundary data by Li \cite{Li18}. 

In higher dimension $n \ge 3$, we consider perturbations in the $W^{1,n-1}$ norm. Since the singular set is a rectifiable set of codimension $3$, we prove its stability with respect to Wasserstein metric, see \cite{Villani}, 
\begin{equation}
\label{eq:wasserstein}
d_W(\mu,\nu) = \sup \left \{ \int_{\R^n} h \dd \mu - \int_{\R^n} h \dd \nu \ : \ h \colon \R^n \to \R, \ |h| \le 1, \ |\nabla h| \le 1 \right \},
\end{equation}
i.e., we show that the distance between measures $\cH^{n-3} \llcorner \sing u$ and $\cH^{n-3} \llcorner \sing u'$ is small. Since taking $h \equiv 1$ in the definition yields 
\[
|\mu(\R^n) - \nu(\R^n)| \le d_W(\mu,\nu),
\]
we obtain in particular that the size of the singular set $\cH^{n-3}(\sing u)$ is also stable under $W^{1,n-1}$-perturbations of boundary data. 

\begin{theorem}[stability of singularities]
\label{th:global-stability}
Let $\Omega \subset \R^n$ be a bounded smooth domain, and let $u \in W^{1,2}(\Omega,\S^2)$ be a minimizer with boundary data $\varphi \in W^{1,n-1}(\partial \Omega,\S^2)$. If $u_k$ is a sequence of minimizers with boundary data $\varphi_k$ and 
\begin{equation}\label{eq:thgs:convergence}
u_k \to u \text{ in } W^{1,2}(\Omega), 
\quad 
\varphi_k \to \varphi \text{ in } W^{1,n-1}(\partial \Omega), 
\end{equation}
then 
\[
\H^{n-3} \llcorner \sing u_k 
\xrightarrow{d_W}
\H^{n-3} \llcorner \sing u,
\]
in particular $\H^{n-3}(\sing u_k) \to \H^{n-3}(\sing u)$. 
\end{theorem}
Under the assumption of uniqueness, we obtain immediately Theorem~\ref{th:stabilitynD}.

\begin{proof}[Proof of Theorem~\ref{th:stabilitynD}]
For the sake of contradiction, let $u_k$ be a sequence of minimizers with boundary data $\vp_k$, with $\vp_k \to \vp$ in $W^{1,n-1}(\partial \Omega, \S^2)$. Taking a subsequence, by Theorem~\ref{th:bsc} we may assume that $u_k$ converges in $W^{1,2}(\Omega,\S^2)$ to a minimizer $\overline{u}$ with boundary data $\vp$. By uniqueness, $\overline{u} = u$ and Theorem~\ref{th:global-stability} implies that $\H^{n-3} \llcorner \sing u_k$ tends to $\H^{n-3} \llcorner \sing u$. Thus, we obtain a contradiction for large enough $k$. 
\end{proof}

\subsection{Outline}

In analogy to the original argument of Hardt and Lin \cite{HL1989}, the heart of the argument lies in the special case when $u$ is the tangent map $\Psi$ as in \eqref{eq:def-psi} given by 
\[
\R^3 \times \R^{n-3} \ni (x',x'')
\xmapsto{\quad \Psi \quad}
\frac{x'}{|x'|} \in \S^2.
\]
Establishing a stability result for the singular set (which for $\Psi$ is an $(n-3)$-dimensional plane) requires some care. Here we adopt the notion of $\delta$-flatness introduced in \cite{Mis18}, which combines topological and analytic conditions for a minimizer to be \emph{close} to $\Psi$. In Section~\ref{sec:stab-flatness} we cite the necessary results and also show that the condition for $\delta$-flatness is stable under $W^{1,2}$-perturbations of the minimizer (Proposition~\ref{prop:stability-of-flatness}). 

With this in hand, we are able to modify the original arguments of Naber and Valtorta \cite{NabVal17} and improve on them in the special case of maps into $\S^2$. In result, we obtain the stability result for $\Psi$ mentioned earlier (Lemma \ref{lem:local-stability}). 

Since around $\mathcal{H}^{n-3}$-almost every singular point, any energy minimizer is close to the map $\Psi$ (composed with an isometry), this stability result can be seen as a local case for Theorem~\ref{th:global-stability}. Indeed, in Section~\ref{sec:stab-global} we cover most of the singular set of $u$ by balls on which Lemma~\ref{lem:local-stability} can be applied. An argument based on Proposition~\ref{prop:stability-of-flatness} then shows that the same covering works for both $\sing u$ and $\sing u_k$, and the global estimate follows. 


\subsection{Behavior of top-dimensional singularities}
\label{sec:stab-flatness}
This subsection gathers the results of \cite{Mis18}, which allow us to study further the top-dimensional part of the singular set. 

Recall the tangent map $\Psi$ from \eqref{eq:def-psi} and its energy density $\Theta$ from \eqref{eq:def-Theta}, and the rescaled energy $\theta_u$ from \eqref{eq:thetau}. We introduce the following property, which basically says that $u$ is close to $\Psi$ (up to an isometry) on the ball $B_r(x)$. 
\begin{definition}[$\delta$-flatness]
\label{df:mis-flatness}
We say that an energy minimizer $u \colon \Omega \to \S^2$ is $\delta$-flat in the ball $\B_r(x) \subset \Omega$ if 
\begin{enumerate}
\item
$x$ is a singular point of $u$ and $\Theta \le \theta_u(x,0) \le \theta_u(x,r) \le \Theta + \delta$, 
\item
for some $(n-3)$-dimensional affine plane $L$ through $x$, $\sing u \cap B_r(x) \subset B_{r/10}(L)$, 
\item
$u$ restricted to $(x+L^\perp) \cap \partial \B_{r/2}(x)$ has degree $\pm 1$ as a map to $\S^2$. 
\end{enumerate}
\end{definition}
Note that this definition is scale-invariant in the following sense: $u$ is $\delta$-flat in $B_r(x)$ if and only if the rescaled map $\overline{u}(y) = u(x+ry)$ is $\delta$-flat in $B_1$. Also note that $u$ is smooth outside the tube around $L$ by and thus the degree is well-defined. 

Definition~\ref{df:mis-flatness} is strongly reminiscent of \cite[Def.~4.3]{Mis18}. There, Reifenberg flattness is additionally assumed, but it follows from 

\begin{lemma}[{{\cite[Lemma~5.1]{Mis18}}}]
\label{lem:mis-forced-sing}
Assume that $\sing u \cap B_r(x) \subset B_{\eps r}(L)$ for some $0 < \eps < \frac 12$ and some $(n-3)$-dimensional plane $L$ through $x$. Moreover, assume that $u$ restricted to $(x+L^\perp) \cap \partial \B_{r/2}(x)$
has degree $\pm 1$ as a map from $\S^2$ to itself. Then 
\[
L \cap B_{(1-\eps)r}(x) \subset \pi_L(\sing u \cap B_r(x)). 
\]
Here and henceforth, $\pi_L$ denotes the nearest-point projection from $\R^n$ onto $L$.
\end{lemma}

In particular, it follows from our definition of $\delta$-flatness that $L \cap B_r(x) \subset B_{r/5}(\sing u)$. This allows us to apply the results of \cite{Mis18}.

The first important point is that around each point in top-dimensional part of the singular set, $\sing_* u$, the map $u$ satisfies the $\delta$-flattness property on sufficiently small balls.

\begin{lemma}[{{\cite[Cor~5.4,~Lem~5.8]{Mis18}}}]
\label{lem:mis-some-flatness}
Let $x \in \sing_* u$. Then for each $\delta > 0$ there is $r_0 > 0$ such that $u$ is $\delta$-flat in $B_r(x)$ for all $r \in (0,r_0]$. 
\end{lemma}

Below we also note various consequences of $\delta$-flatness proved in \cite{Mis18}. For simplicity, we only deal with the unit ball, but one can easily obtain the corresponding statement for any ball using the scale-invariance. 

\begin{theorem}
\label{th:mis-all}
For each $\eps > 0$ there is $\delta > 0$ such that the following holds. If $u$ is $\delta$-flat in $B_1$, then 
\begin{enumerate}

\item for some tangent map of the form $\overline{\Psi} = \Psi \circ \tau$ (with $\Psi$ as in \eqref{eq:def-psi} and some linear isometry $\tau$) we have 
\[
\| u - \overline{\Psi} \|^2_{W^{1,2}(B_1)} \le \eps, 
\]

\item for the $(n-3)$-dimensional linear plane $L' := \sing \overline{\Psi}$, 
\[
\sing u \cap B_1 \subset B_\eps (L')
\quad \text{and} \quad
L' \cap B_{1-\eps} \subset \pi_{L'}(\sing u \cap B_1), 
\]

\item all singular points in $B_{1/2}$ lie in the top-dimensional part $\sing_* u$, and $u$ is $\eps$-flat in each of the balls $B_r(z)$ with $z \in \sing u \cap B_{1/2}$ and $0 < r \le 1/2$. 

\end{enumerate}
\end{theorem}

\begin{proof}
Due to Lemma~\ref{lem:mis-forced-sing}, we may apply the results of \cite{Mis18} directly. 

Points (1) and (2) are essentially the content of \cite[Lem~5.3]{Mis18}, except for the condition $L \cap B_{1-\eps} \subset \pi_L(\sing u \cap B_1)$, which again follows from Lemma~\ref{lem:mis-forced-sing}. Point (3) comes from combining \cite[Prop~5.6]{Mis18} and its corollary \cite[Cor~5.7]{Mis18}. 
\end{proof}

The last ingredient is another consequence of the arguments in \cite{Mis18}. It is to some extent the higher-dimensional analogue of \cite[Theorem 1.8, (2)]{AlmgrenLieb1988} (see \cite[Theorem 4.8 (2)]{MMS18}). 

\begin{proposition}[Stability of $\delta$-flatness]
\label{prop:stability-of-flatness}
For each $\eps > 0$ there is $\delta > 0$ such that the following holds. If $u$ is $\delta$-flat in the ball $B_1$ and $u_k \xrightarrow{k \to \infty} u$ in $W^{1,2}(B_1)$, then for $k$ large enough there is $x_k \in \sing u_k \cap B_\eps$ such that $u_k$ is $\eps$-flat in the ball $B_{1-\eps}(x_k)$. 
\end{proposition}

\begin{proof}
Choose $\eps'(n,\eps) > 0$ small enough, more precisely such that 
\[
\eps' < \eps/2, 
\quad 
(1-2\eps')^{2-n} (\Theta + \eps/2) \le \Theta + \eps.
\]
By taking $\delta$ small enough, we may assume by Theorem~\ref{th:mis-all} that
\[
\sing u \cap B_1 \subset B_{\eps'/2} (L) 
\]
for some $(n-3)$-dimensional linear plane $L$. 
Since singular points converge again to singular points, Theorem~\ref{th:ALs2s}, we have for all large $k$,
\begin{equation}\label{eq:ukclosetoL}
\sing u_k \cap B_{1-\eps'} \subset B_{\eps'/2} (L) 
\end{equation}
By \cite[Proposition 4.6]{SU1}, we have locally uniform convergence outside the singular set, and thus
\[
u_k \rightrightarrows u \quad \text{ in } B_{1-\eps'} \setminus B_{\eps'/2}(L).
\] 
In particular, $u_k$ and $u$ restricted to $L^\perp \cap \partial B_{1/2}$ have the same homotopy type for large $k$.

By Lemma~\ref{lem:mis-forced-sing}
\[
 L \cap B_{1-2\eps'} \subset  \pi_L(\sing u_k \cap B_{1-\eps'}).
\]
Combined with \eqref{eq:ukclosetoL} this means that $u_k$ has many singular points near $L$. Since $\H^{n-3}$-a.e. singular point lies in $\sing_* u$ (see \eqref{eq:Sjdim}), we find $x_k \in \sing_* u_k$ with $|x_k| \le \frac{1}{2}\eps'$. In particular, we already have $\theta_{u_k}(x_k,0) = \Theta$, by Corollary~\ref{co:uniquetangent}. 

The last condition to show is
$\theta_{u_k}(x_k,1-\eps) \le \Theta + \eps$. By strong convergence, for large enough $k$,
\[
\int_{B_{1-\eps'}} |\nabla u_k|^2 \le \eps/4 + \int_{B_1} |\nabla u|^2.
\]
Thus 
\[
(1-2\eps')^{2-n} \int_{B_{1-2\eps'}(x_k)} |\nabla u_k|^2 
\le (1-2\eps')^{2-n} \left( \eps/4 + \int_{B_1} |\nabla u|^2 \right) 
\le (1-2\eps')^{2-n} (\Theta + \delta + \eps/4), 
\]
which does not exceed $\Theta + \eps$ if only $\delta \le \eps/4$. By the monotonicity formula, we conclude that $\theta_{u_k}(x_k,1-\eps) \le \theta_{u_k}(x_k,1-2\eps') \le \Theta + \eps$ and hence that $u_k$ is $\eps$-flat in the ball $B_{1-\eps}(x_k)$.
\end{proof}

\subsection{Local case}
\label{sec:stab-local}

The lemma below can be thought of as a local version of the stability theorem. It says that perturbing the tangent map $\Psi$ a little does not change the size of the singular set much. 

\begin{lemma}
\label{lem:local-stability}
For each $\eps > 0$ there is $\delta > 0$ such that the following is true. If $u \colon B_{2}(0) \to \S^2$ is energy minimizing and $\delta$-flat in $B_{2}(0)$ (see Definition \ref{df:mis-flatness}), then 
\[
(1-\eps) \omega_{n-3} \le \H^{n-3}(\sing u \cap B_1) \le (1+\eps) \omega_{n-3}. 
\]
Here $\omega_{n-3} = \H^{n-3}(\sing \Psi \cap B_1) $ is the volume of the $(n-3)$-dimensional ball.
\end{lemma}

\begin{proof}
While the ball $B_2$ is indeed enough here, we will work with $B_{40}$ which saves us from an additional covering argument. 

The lower bound follows from a simple topological argument. Fix $\eps' = \frac{\eps}{n-2}$, then apply Theorem~\ref{th:mis-all} to find that there is an $(n-3)$-dimensional linear plane $L$ such that 
\[
L \cap B_{1-\eps'} \subset \pi_L( \sing u \cap B_1 ),
\]
provided $\delta$ is small enough. Since the orthogonal projection $\pi_L$ is $1$-Lipschitz, this shows 
\[
\H^{n-3}(\sing u \cap B_1) \ge \H^{n-3}(L \cap B_{1-\eps'}) = (1-\eps')^{n-3} \omega_{n-3} \ge (1-\eps) \omega_{n-3}. 
\]

A rough upper bound follows from Naber and Valtorta's work \cite{NabVal17}, namely Corollary~\ref{co:NabVal-meas-bound},
\begin{equation}
\label{eq:weak-meas-bound}
\H^{n-3}(\sing u \cap B_r(z)) \le C(n) r^{n-3}
\end{equation}
for each ball $B_{2r}(z) \subset B_{2}$.

To obtain the sharp upper bound, we will follow the general outline of Naber and Valtorta's work \cite[Sec.~1.4]{NabVal17}. When the target manifold is $\S^2$, the original reasoning can be made significantly easier due to topological control of singularities (analyzed in \cite{Mis18}). In particular, we we will be able to apply Rectifiable Reifenberg Theorem~\ref{th:NabVal-reifenberg} to the whole singular set in $B_1$, without decomposing it into many pieces. 

With $\delta_1 > 0$ to be fixed later, by Theorem~\ref{th:mis-all} we can choose $\delta$ small enough so that all singular points in $B_{20}$ lie in the top-dimensional part $\sing_* u$, moreover $u$ is also $\delta_1$-flat in each ball $B_{r}(z)$ with $z \in \sing u \cap B_{20}$ and $0 < r \le 20$. 

We can now apply the $L^2$-best approximation Theorem~\ref{th:NabVal-best-approx} on these balls; for simplicity, we consider the ball $B_{10}$ first. By Theorem~\ref{th:mis-all}, $u$ is $W^{1,2}$-close to a map of the form $\overline{\Psi} = \Psi \circ q$ (with $\Psi$ as in \eqref{eq:def-psi} and some linear isometry $\tau$). Note that $\overline{\Psi}$ lies in $\mathrm{sym}_{n,0}$ and the value 
\[
\eps_0 := \dist_{L^2(B_{10})} (\overline{\Psi}, \ \mathrm{sym}_{n,k+1}) > 0
\]
depends only on the dimension $n$ (not on the choice of $\tau$). Hence, by taking $\delta_1$ small enough we can ensure that 
\begin{align*}
\dist_{L^2(B_{10})} (u, \ \mathrm{sym}_{n,0}  ) & \le \delta, \\
\dist_{L^2(B_{10})} (u, \ \mathrm{sym}_{n,k+1}) & \ge 2\eps_0 
\end{align*}
with $\delta=\delta(\eps_0)$ chosen according to Theorem~\ref{th:NabVal-best-approx}. Then we obtain 
\[
\beta (0,1)^2
\le C(n) \int_{B_1} \left( \theta_u(y,8) - \theta_u(y,1) \right) \dd \mu(y),
\]
where $\mu := \H^{n-3} \llcorner \sing u$ and $\beta = \beta_{\mu,n-3,2}$. Similarly, 
\begin{equation}
\label{eq:beta-estimate}
\beta (z,s)^2
\le C(n) s^{-(n-3)} \int_{B_s(z)} \left( \theta_u(y,8s) - \theta_u(y,s) \right) \dd \mu(y) 
\end{equation}
for each ball $B_s(z) \subset B_2$ with $z \in \sing u$. To see this, one simply needs to consider the rescaled map $\overline{u}(x) = u(z+rx)$ and apply scaling-invariance of $\delta$-flatness and $\beta$-numbers. 


Now we verify the hypotheses of Rectifiable Reifenberg Theorem~\ref{th:NabVal-reifenberg}. Fix a ball $B_r(x) \subset B_2$; we only need to check that 
\begin{equation}
\label{eq:jones-estimate}
\int_{B_r(x)} \int_0^r \beta(z,s)^2 \frac{\dd s}{s} \dd \mu(z) \le \delta_2 r^{n-3} 
\end{equation}
with $\delta_2(\eps) > 0$ chosen according to Theorem~\ref{th:NabVal-reifenberg}, 

First, we integrate the estimate \eqref{eq:beta-estimate} over $B_r(x)$ and exchange the order of summation: 
\begin{align*}
\int_{B_r(x)} \beta(z,s)^2 \dd \mu(z) 
& \lesssim s^{-(n-3)} \int_{B_r(x)} \int_{B_s(z)} (\theta_u(y,8s)-\theta_u(y,s)) \dd \mu(y) \dd \mu(z) \\
& \le s^{-(n-3)} \int_{B_{2r}(x)} \int_{B_s(y)} (\theta_u(y,8s)-\theta_u(y,s)) \dd \mu(z) \dd \mu(y) \\
& \lesssim \int_{B_{2r}(x)} (\theta_u(y,8s)-\theta_u(y,s)) \dd \mu(y)
\end{align*}
Note that in the last step we used the weak upper bound \eqref{eq:weak-meas-bound} on the ball $B_s(y)$. 

When the above is integrated with respect to $s$, we obtain a telescopic sum. In order to estimate it, first recall that $u$ is $\delta_1$-flat in each ball $B_{8s}(y)$ with $y \in \sing u \cap B_{20}$ and $0 < s \le 2$, in particular 
\[
\theta_u(y,8r) - \theta_u(y,0) \le \delta_1 
\]
on the support of $\mu$. Thus, the substitution $s \mapsto 8s$ together with monotone convergence $\theta_u(y,s) \searrow \theta_u(y,0)$ give us 
\begin{align*}
\int_0^r (\theta_u(y,8s)-\theta_u(y,s)) \frac{\dd s}{s} 
& = \int_{r}^{8r} (\theta_u(y,s) - \theta_u(y,0)) \frac{\dd s}{s} \\
& \le \ln(8) \delta_1.
\end{align*}
Now we are ready to combine the above estimates: 
\begin{align*}
\int_{B_r(x)} \int_0^r \beta_{\mu,2}^2(z,s) \frac{\dd s}{s} \dd \mu(z) 
& \lesssim \int_0^r \int_{B_{2r}(x)} (\theta_u(y,8s)-\theta_u(y,s)) \dd \mu(y) \frac{\dd s}{s} \\
& \le \int_{B_{2r}(x)} \ln(8) \delta_1 \dd \mu(y) \\
& \lesssim \delta_1 r^{n-3},
\end{align*}
where we used \eqref{eq:weak-meas-bound} again in the last line. Assuming $\delta_1 \le \delta_2(\eps) / C(n)$, we have verified the assumption \eqref{eq:jones-estimate} and we infer the upper estimate 
\[
\H^{n-3}(\sing u \cap B_1) = \mu(B_1) \le (1+\eps) \omega_{n-3}. 
\] 

\end{proof}

\subsection{Global case}
\label{sec:stab-global}
The idea of the proof is to cover most of $\sing u$ by good balls, on which $u$ is $\delta$-flat and thus the measure of $\sing u$ is controlled by Lemma \ref{lem:local-stability}. The rest of the singular set is to be covered by bad balls, whose total mass is small. To achieve this, we will need the following simple covering lemma. 

\begin{lemma}
\label{lem:easy-covering}
For $k \geq 0$, let $S \subset \R^n$ be a compact set of finite $\cH^{k}$-measure. 

Let $\cB$ be a family of open balls with the following property:
For all $p \in S$ there exists $r(p) > 0$ such that $B_r(p) \in \cB$ for all $r < r(p)$.

Then, given any $\eps > 0$, $S$ can be covered by the union of families of open balls $\good$, $\bad$, where $\good \subset \cB$ consists of finitely many pairwise disjoint balls and $\bad = B_{r_j}(p_j)$ is a countable family which is small in the sense that 
\begin{equation}
\label{eq:cov:smallness}
\sum_j r_j^{k} \le \eps. 
\end{equation}
\end{lemma}

\begin{proof}
One way to construct this covering is by using Vitali's covering theorem for Radon measures, e.g., \cite[Theorem 2.8]{Mat95}. Applying it to the measure $\mu := \cH^{k} \llcorner S$, we obtain a countable family of pairwise disjoint balls $\cA = \left\{ \overline{B_{r_s}(p_s)} \right\}$, covering $\mu$-almost all $S$ and satisfying $B_{2r_s}(p_s) \in \cB$ for each $s$. Since the series $\sum_s \mu(B_{r_s}(p_s))$ converges, we can divide $\cA$ into two subfamilies $\good'$, $\bad'$, where $\good'$ is finite and $\bad'$ is small, i.e., $\mu \left( \bigcup \bad' \right) \le \eps$. To obtain the desired properties, we still need to alter these families a little. 

First, we define $\good$ to be the balls of $\good'$ slightly enlarged to open balls, but still pairwise disjoint and still belonging to $\cB$. 

Now, the remaining part $S \setminus \bigcup \good$ is a compact set and 
\[
\mu \left( S \setminus \bigcup \good \right) \le \mu \left( \bigcup \bad' \right) \le \eps.
\]
By definition of Hausdorff measure, this set can be covered by a countable family of open balls $\bad$ satisfying the smallness condition \eqref{eq:cov:smallness}. 
\end{proof}

\begin{proof}[Proof of Theorem~\ref{th:global-stability}]

Fix $\varepsilon > 0$. For the sake of clarity, we focus on showing that the difference $| \H^{n-3}(\sing u_k) - \H^{n-3}(\sing u) |$ is controlled by $\eps$ for $k$ large enough. The estimate for Wasserstein distance follows the same lines; it is briefly discussed at the end of the proof. 

\textsc{Step 1 (boundary regularity).}
Choose $\varepsilon_0 > 0$ according to the boundary regularity theorem, Theorem~\ref{th:r4s}. Fix $\rho > 0$ such that 
\[
\sup_{x \in \partial \Omega} \int_{B_\rho(x)} |\nabla \varphi|^{n-1} \le \varepsilon_0 / 2.
\]
Then $u$ is smooth in a $\lambda \rho$-neighborhood of $\partial \Omega$. By strong convergence of $\varphi_k$ to $\varphi$ in $W^{1,n-1}(\partial \Omega)$, we may assume w.l.o.g. for all $k \in \N$,
\[
\sup_{k} \sup_{x \in \partial \Omega} \int_{B_\rho(x)} |\nabla \varphi_k|^{n-1} \le \varepsilon_0.
\]
As a consequence, we may assume each $u_k$ is also smooth in the same fixed neighborhood of $\partial \Omega$. 

\textsc{Step 2 (covering the low-dimensional part).}
Recall the stratification, Section~\ref{s:tangentmaps},
\[
S_0 \subset \ldots \subset S_{n-4} \subset S_{n-3} = \sing u,
\]
in which the $k$-th stratum $S_k$ has Hausdorff dimension $k$ or smaller. We will consider separately the set $S_{n-4}$ and the top-dimensional part 
\[
\sing_* u := S_{n-3} \setminus S_{n-4}.
\]
Since $\sing u$ is compact and $\sing_* u$ is an open subset of $\sing u$ (see Theorem~\ref{th:mis-all}), $S_{n-4}$ is also compact. At the same time, it has a uniform distance from $\partial \Omega$ and $\H^{n-3}(S_{n-4})=0$, so it can be covered by a finite family $\bad_1 = \{ B_{r_i}(p_i) \}$ of open balls satisfying the smallness condition \eqref{eq:cov:smallness} 
\[
\sum_i r_i^{n-3} \le \varepsilon 
\]
and such that $B_{2r_i}(p_i) \subset \Omega$ for each $i$. 

On each such ball Corollary~\ref{co:NabVal-meas-bound} yields $\H^{n-3}(\sing u \cap B_{r_i}(p_i)) \le C r_i^{n-3}$, with $C$ depending only on the dimension $n$. Summing over all balls, we obtain 
\[
\H^{n-3} \left( \sing u \cap \bigcup \bad_1 \right) 
\le C \eps. 
\]
The same estimate holds verbatim for each $u_k$, by the same application of Corollary \ref{co:NabVal-meas-bound}. 

\textsc{Step 3 (covering the top-dimensional part and estimating $\H^{n-3}(\sing u)$).}
Here, we use the covering lemma (Lemma \ref{lem:easy-covering}) for the set $S := \sing u \setminus \bigcup \bad_1$. Thanks to Step 1, $\sing u$ has positive distance from the boundary, so it is a compact set of finite $\cH^{n-3}$-measure due to Corollary \ref{co:NabVal-meas-bound}. We choose $\cB$ to be 
\[
\cB = \left \{ B_r(p) \colon \quad p \in \sing_* u, \text{$u$ is $\delta$-flat in $B_{41r}(p)$} \right \},
\]
where $\delta(\eps) > 0$ is chosen according to Lemma \ref{lem:local-stability}. Since $S_{n-4}$ is already covered by $\bad_1$, we know that $S \subset \sing_* u$ and hence small enough balls around each point in $S$ lie in $\cB$ by Lemma \ref{lem:mis-some-flatness}. 

Having checked the properties required by Lemma \ref{lem:easy-covering}, we can cover $S$ by the union of a finite disjoint family $\good \subset \cB$ and another countable family $\bad_2$ satisfying \eqref{eq:cov:smallness}. We add the latter to $\bad_1$ to obtain the family of bad balls $\bad := \bad_1 \cup \bad_2$, which still satisfies the smallness condition \eqref{eq:cov:smallness}. 

Repeating the reasoning from Step 2, we have again via Corollary~\ref{co:NabVal-meas-bound},
\begin{equation}
\label{eq:comparison-bad-balls}
\H^{n-3} \left( \sing u \cap \bigcup \bad \right) 
\le 2C \eps, 
\quad  
\H^{n-3} \left( \sing u_k \cap \bigcup \bad \right) 
\le 2C \eps 
\quad \text{for all } k.
\end{equation}

By assumption, the map $u$ is $\delta$-flat in $B_{40r_s}(p_s)$ for each ball $B_{r_s}(p_s) \in \good$. By Lemma~\ref{lem:local-stability}, we now obtain 
\[
(1-\eps) \omega_{n-3} r_s^{n-3} 
\le \H^{n-3}(\sing u \cap B_{r_s}(p_s)) 
\le (1+\eps) \omega_{n-3} r_s^{n-3} 
\]
for each $s$. To finish the proof, we need to show that a similar comparison holds for $u_k$ if $k$ is large.

\textsc{Step 4 (estimating $\H^{n-3}(\sing u_k)$).} Since $u_k \to u$ in $W^{1,2}(\Omega)$ and $\sing u$ is covered by the open families $\good,\bad$, Theorem~\ref{th:ALs2s} (singular points converge to singular points) implies that the same holds for $u_k$ if $k$ is large enough (from now on we assume it is). For bad balls, the rough estimate \eqref{eq:comparison-bad-balls} will be enough, so we focus on good balls. 

By Proposition~\ref{prop:stability-of-flatness}, we can assume (by taking $k$ large and $\delta$ small) that for each $B_{r_s}(p_s) \in \good$ there is $p^k_s \in \sing u_k$ such that $|p^k_s-p_s| \le \eps r_s$ and $u_k$ is $\delta'$-flat in the ball $B_{40(1+\eps)r_s}(p^k_s)$. Here, the value of $\delta'$ is chosen to be $\delta(\eps)$ from Lemma~\ref{lem:local-stability}. 

Applying Lemma~\ref{lem:local-stability} to $u_k$ on balls $B_{(1-\eps)r_s}(p^k_s)$ and $B_{(1-\eps)r}(p^k_s)$, we obtain  
\begin{align*}
(1-\eps)^{n-2} \omega_{n-3} r_s^{n-3} 
& \le \H^{n-3}(\sing u_k \cap B_{(1-\eps)r_s}(p^k_s)) \\
& \le \H^{n-3}(\sing u_k \cap B_{r_s}(p_s)) \\
& \le \H^{n-3}(\sing u_k \cap B_{(1+\eps)r_s}(p^k_s)) \\
& \le (1+\eps)^{n-2} \omega_{n-3} r_s^{n-3}, 
\end{align*}
which is only slightly worse that the estimate for $\H^{n-3}(\sing u)$.

\textsc{Step 4 (comparison).}
Recalling that $\good$ is a disjoint family, we can sum the above estimate over all $s$ to obtain 
\[
(1-\eps)^{n-2} A 
\le \H^{n-3}(\sing u_k \cap \bigcup \good) 
\le (1+\eps)^{n-2} A,  
\]
where $A := \sum_s \omega_{n-3} r_s^{n-3}$. 
Combining it with the estimate for bad balls \eqref{eq:comparison-bad-balls}, we finally obtain 
\[
(1-\eps)^{n-2} A 
\le \H^{n-3}(\sing u_k)
\le (1+\eps)^{n-2} A + 2C \eps.
\]
Exactly the same estimate is true for $u$. Combining these two yields 
\begin{align*}
\big| \H^{n-3}(\sing u_k) - \H^{n-3}(\sing u) \big| 
& \le \left( (1+\eps)^{n-2} - (1-\eps)^{n-2} \right) A + 2C \eps \\
& \le \left( \frac{(1+\eps)^{n-2}}{(1-\eps)^{n-2}} - 1 \right) \H^{n-3}(\sing u) + 2C \eps. 
\end{align*}
Evidently the right-hand side tends to zero when $\eps \to 0$, which ends the proof of stability of $\cH^{n-3}(\sing u)$. 

\textsc{Step 5 (Wasserstein distance estimate).} With just a little bit more care, the Wasserstein distance estimate follows. Let us decompose the measure $\mu := \cH^{n-3} \llcorner \sing u$ into $\mu = \mu_b + \sum_s \mu_s$, where 
\[
\mu_b = \mu \llcorner \left( \bigcup \bad \setminus \bigcup \good \right), \quad \mu_s = \mu \llcorner B_{r_s}(p_s) \quad \text{for each ball } B_{r_s}(p_s) \in \good.
\]
The estimate for $\mu_b$ is simply $d_W(\mu_b, 0) \le \mu \left ( \bigcup \bad \right) \le 2 C \eps$, whereas on each good ball $B_{r_s}(p_s)$ we have the inequalities 
\begin{align*}
\int_{\R^n} h \dd \mu_s - \omega_{n-3} r_s^{n-3} h(p_s) 
& = \int_{B_{r_s}(p_s)} (h - h(p_s)) \dd \mu + (\mu(B_{r_s}(p_s)) - \omega_{n-3} r_s^{n-3}) h(p_s) \\
& \le r_s \mu(B_{r_s}(p_s)) + |\mu(B_{r_s}(p_s)) - \omega_{n-3} r_s^{n-3}| \\
& \le (r_s + 2\eps) \omega_{n-3} r_s^{n-3}. 
\end{align*}
for any function $h \colon \R^n \to \R$ satisfying $|h| \le 1$ and $|\nabla h| \le 1$. Thus $d_W(\mu_k, \omega_{n-3} r_s^{n-3} \delta_{p_s}) \le 3\eps \omega_{n-3} r_s^{n-3}$, if only each radius is smaller than $\eps$. By triangle inequality, $d_W(\mu,\nu) \le 3 \eps A + 2 C \eps$, where $\nu = \sum_s \omega_{n-3} r_s^{n-3} \delta_{p_s}$ is the packing measure associated to $\good$ and once again $A = \nu(\R^n)$. Applying the same reasoning to $u_k$, we conclude as before. 

\end{proof}

\section{Almgren and Lieb's linear law for \texorpdfstring{$n \geq 3$}{n >= 3}: size of the singular set}
Here we obtain following higher-dimensional counterpart for Almgren--Lieb's linear estimate on the number of singularities. 
Let us stress, that the fundamental result that makes such estimates possible is Naber and Valtorta's breakthrough paper \cite{NabVal17}, Corollary~\ref{co:NabVal-meas-bound}.

\begin{theorem}
\label{th:almgren-lieb-high}
 Let $u\in W^{1,2}(\Omega,\S^2)$ be a minimizing map with $u\rvert_{\partial \Omega} = \varphi$, $\varphi \colon \partial \Omega \to \S^2$, where $\Omega\subset \R^n$ is an open bounded domain. Then 
 \begin{equation}\label{eq:mainestimate}
  \H^{n-3}(\sing u) \le C(\Omega) \int_{\partial \Omega} |\nabla \varphi|^{n-1} \dhn.
 \end{equation}
\end{theorem}

As in the case $n=3$, the study of singularities near the boundary involves the following covering lemma, which we here cite from \cite[Theorem 2.8, 2.9]{AlmgrenLieb1988}. 

\begin{theorem}[Covering lemma]
\label{th:covering}
Let $\mathcal B$ be a family of closed balls in $\R^n$, $\mu$ be a Borel measure over $\R^n$, and let $\tau,\omega \in (0,1)$. Moreover, assume that the following two hypotheses hold: 
\begin{enumerate}
 \item For any two different $B_r(p), B_s(q) \in \mathcal{B}$ we have 
 \[
  |p-q| \ge \omega \min(r,s).
 \]
 \item Suppose that $B_r(p)\in \cB$ and $q\in\R^n$ is an arbitrary point, then 
 \[
  \mu \left( B_r(p) \setminus B_{\tau r}(q) \right) \ge 1.
 \]
\end{enumerate}
Then
\[
\# \text{balls in }\mathcal{B} \le C \mu(\R^n), 
\]
for a constant $C(\omega,\tau,n) > 0$. 
\end{theorem}

\begin{proof}[Proof of Theorem~\ref{th:almgren-lieb-high}]
Choose $\sigma > 0$ (depending on the geometry of $\partial \Omega$) according to Theorems \ref{th:int-regularity-in-terms-of-bdry}, \ref{th:hot-spots}. We first estimate the measure of the set 
\[
 A_1 := \{ p \in \sing u : r(p) \le \sigma \}, 
 \quad \text{where } r(p) = \tfrac 12 \dist(p,\partial \Omega),  
\]
which is covered by balls $B_{r(p)}(p)$. Then choose a Vitali subcovering such that the balls $B_{r_j}(p_j)$ cover $A_1$ and the balls $B_{r_j/5}(p_j)$ are disjoint; let $\cB$ be the family of balls $B_{r_j/\lambda}(p_j)$ with $\lambda$ as in Theorem~\ref{th:hot-spots}. The first condition from Theorem~\ref{th:covering} with $\omega = \lambda/5$ follows: for any two distinct balls in our collection we have 
\[
 |p_i-p_j| \ge \tfrac 15 (r_i+r_j) \ge \tfrac{\lambda}{5} \max(r_i/\lambda,r_j/\lambda).
\]
Now let $\mu$ be the measure 
\[
 \mu = \frac 1 \eps |\nabla \varphi|^{n-1} \, \cH^{n-1} \llcorner \partial \Omega, 
 \quad \text{i.e. } 
 \mu(U) = \frac 1 \eps \int_{\partial \Omega \cap U} |\nabla \varphi|^{n-1} \dd \cH^{n-1},
\]
where $\eps > 0$ is the constant from Theorem~\ref{th:hot-spots}. If we set $\tau = \lambda^2$, then the second condition of Theorem~\ref{th:covering} with $k = n-3$ follows from Theorem~\ref{th:hot-spots} and we infer that 
\[
\# \cB \le C \int_{\partial \Omega} |\nabla \varphi|^{n-1} \dd \cH^{n-1}. 
\]
On each ball $B_{r_j}(p_j)$ Corollary~\ref{co:NabVal-meas-bound} implies $\H^{n-3}(\sing u \cap B_{r_j}(p_j)) \le C r_j^{n-3} C(\Omega)$. Summing over all balls, we obtain 
\[
\H^{n-3}(A_1) \le C \int_{\partial \Omega} |\nabla \varphi|^{n-1} \dd \cH^{n-1}. 
\]

Next we estimate the set 
\[
 A_2 := \{ p \in \sing u : r(p) \ge \sigma \}.
\]
For each ball $B_\sigma(y)$ with $\dist(y,\partial \Omega) \ge 2 \sigma$ we have a bound $\H^{n-3}(\sing u \cap B_\sigma(y)) \le C(\Omega)$ by Corollary~\ref{co:NabVal-meas-bound}. The set $A_2$ can be covered by finitely many such balls (the number of balls depending only on $\sigma$ and the geometry of $\Omega$), which gives us an estimate 
\[
 \H^{n-3}(A_2) \le C_0. 
\]
Taking $C_0$ as above and $\eps$ as in Theorem~\ref{th:int-regularity-in-terms-of-bdry}, we see that either $\int_{\partial \Omega} |\nabla \varphi|^{n-1} \dhn \le \eps$ and $\H^{n-3}(A_2) = 0$, or 
\[
\H^{n-3}(A_2) \le C_0 \le \frac{C_0}{\eps} \int_{\partial \Omega} |\nabla \varphi|^{n-1} \dd \cH^{n-1},
\]
which ends the proof.
\end{proof}

\bibliographystyle{abbrv}%
\bibliography{bib}%

\end{document}